\documentclass[a4paper,10pt]{article}
\usepackage[utf8x]{inputenc}
\usepackage{amssymb}
\usepackage{multicol}
\usepackage[inline]{enumitem}
\usepackage{amsmath}
\usepackage{amsthm}
\usepackage{mathrsfs}
\usepackage{a4wide}
\usepackage{cite}
\usepackage[english]{babel}
\usepackage{nicefrac}

\title{Quantitative Strong Convergence for the Hybrid Steepest Descent Method}
\author{ Daniel K\"ornlein\thanks{The 
author has been 
supported by the German Science Foundation (DFG 
Project KO 1737/5-2).}\\[0.2cm]
Department of Mathematics \\  Technische Universit\"at Darmstadt\\ 
Schlossgartenstra\ss{}e 7, 64289 Darmstadt, Germany}

\newcommand{\placeholder}{\delta}

\newtheorem{theorem}{Theorem}[section]
\newtheorem{proposition}[theorem]{Proposition}
\newtheorem{corollary}[theorem]{Corollary}
\newtheorem{lemma}[theorem]{Lemma}

\theoremstyle{definition}

\newtheorem{definition}[theorem]{Definition}

\newtheorem{prop}[theorem]{Proposition}
\newtheorem{prob}[theorem]{Problem}
\newtheorem{notation}[theorem]{Notation}

\theoremstyle{remark}

\newtheorem{remark}[theorem]{Remark}

\newcommand{\NN}{\mathbb{N}}
\newcommand{\RR}{\mathbb{R}}
\newcommand{\F}{\mathcal{F}}
\newcommand{\G}{\mathcal{G}}
\newcommand{\vip}{\operatorname{VIP}}
\newcommand{\fix}{\mathop{Fix}}
\newcommand{\ra}{\rightarrow}
\newcommand{\maj}{\gtrsim}

\newcommand{\diam}{\operatorname{diam}}
\newcommand{\ie}{i.e.~}
\newcommand{\bigand}[2]{\bigwedge\!\!\!\!\!\!\bigwedge_{\!\!\!#1}^{\!\!\!#2}}

\begin{document}

\maketitle

\begin{abstract}
 We provide new complexity information for the convergence of the Hybrid Steepest Descent Method for solving the Variational Inequality Problem for a strict contraction on Hilbert space over a closed convex set $C$ given either as the fixed point set of a single nonexpansive mapping or the intersection of the fixed point sets of a finite family of nonexpansive mappings. More precisely, we give metastability rates in the sense of Tao for those cases. The results in this paper were extracted from a proof due to Yamada using proof-mining techniques, and provide a thorough quantitative analysis of the Hybrid Steepest Descent Method.
\end{abstract}

\section{Introduction}
For a real Hilbert space $H$ and a mapping $\Theta:H\to\RR$, the convex optimization problem for $\Theta$ over some closed convex set $C$ consists in finding a point $x^*$ that minimizes $\Theta$ over $C$. Solving this minimization problem is equivalent to solving the \emph{Variational Inequality Problem} for the gradient $\F:=\Theta'$ over $S$, which is defined as follows:
\begin{equation*}
 \text{Find }u^*\in S\text{ such that }\langle v-u^*,\F u^*\rangle\ge0\text{ for all }v\in S.\tag{$\vip(\F,S)$}
\end{equation*}
Apart from their connection to the convex optimization problem, variational inequalities have numerous applications and have therefore been widely studied in the literature \cite{BarbuPrecupanu,EkelenadIvar(99),ZeidlerIIB,ZeidlerIII}. For an overview of some applications of variational inequalities and the hybrid steepest descent method in particular, we refer the reader to \cite{Yamada(01),YamadaOguraShirakawa(02)}

Apart from existence and uniqueness of solutions, considerable effort has also been put into devising explicit algorithms to compute solutions. The following observation is of central importance for the latter as it transforms the Variational Inequality Problem into a fixed point problem:
\begin{prop}[VIP as a fixed point problem]\label{prop:char}
 Given a mapping $\F:H\to H$ and a nonempty closed and convex set $S$, the following three statements are equivalent.
 \begin{enumerate}[label=(\roman*)]
  \item $u^*\in C$ is a solution to $\vip(\F,S)$, \ie
  \begin{equation*}
   \langle v-u^*,\F(u^*)\rangle\ge0,\quad\textmd{for all }v\in S.
  \end{equation*}
  \item For any $\mu>0$,
  \begin{equation*}
   \langle v-u^*,(u^*-\mu\F(u^*))-u^*\rangle\le0,\quad\textmd{for all }v\in S.
  \end{equation*}
  \item\label{item:characterisation} For any $\mu>0$,
  \begin{equation*}
   u^*\in\fix(P_S(I-\mu\F)).
  \end{equation*}
 \end{enumerate}
\end{prop}

Whenever the mapping $I-\mu\F$ becomes a strict contraction for some $\mu>0$, the map $P_S(I-\mu\F)$ also becomes a contraction, so the Variational Inequality Problem in this situation has a unique solution, and a natural candidate to approximate this solution is the Picard iteration:
\begin{equation*}
 x_{n+1}:=P_S(x_n-\mu\F(x_n)).
\end{equation*}
This algorithm is referred to as the projected gradient method \cite{LevitinPolyak(66),Goldstein(64)} and converges strongly for all $\mu>0$ such that $I-\mu\F$ is a strict contraction. It is known \cite{Yamada(01)} that, if $\F$ is $\kappa$-Lipschitzian and $\eta$-strongly monotone, i.e.
\begin{equation*}
 \langle\F x-\F y,x-y\rangle\ge\eta\Vert x-y\Vert^2,\quad\text{for all }x,y\in H,
\end{equation*}
then $I-\mu\F$ is a strict contraction with Lipschitz constant $\tau:=1-\sqrt{1-\mu(2\eta-\mu\kappa^2)}$ for all $\mu\in(0,2\eta/\kappa^2)$. The main drawback of this approach is that it requires a closed-form expression for the projection $P_S$ onto $S$, which is not always available.

The hybrid steepest descent method \cite{Yamada(01),XuKim(03)}, HSDM for short, avoids the use of the projection $P_S$. This method only requires that the set $S$ is the set of fixed points $\fix(T)$ of some nonexpansive mapping $T:H\to H$:
\begin{theorem}[Yamada \cite{Yamada(01)}]\label{thm:Yam1}
 Let $T:H\to H$ be a nonexpansive mapping with $\fix(T)\ne\emptyset$. Suppose that a mapping $\F:H\to H$ is $\kappa$-Lipschitzian and $\eta$-strongly monotone over $T(H)$. Then, for any $u_0\in H$, any $\mu\in(0,2\eta/\kappa^2)$ and any sequence $(\lambda_n)\subset(0,1]$ satisfying
 \begin{center}
 \begin{enumerate*}
  \item $\displaystyle\lim_{n\to\infty}\lambda_n=0$,\quad\phantom{x}
  \item $\displaystyle\sum_{n=1}^{\infty}\lambda_n$ diverges, and\quad\phantom{x}
  \item $\displaystyle\lim_{n\to\infty}\frac{\lambda_n-\lambda_{n+1}}{\lambda_n^2}=0$,
 \end{enumerate*}
 \end{center}

 \noindent the sequence $(u_n)$ generated by
 \begin{equation*}
  u_{n+1}:=T(u_n)-\lambda_{n+1}\mu\F(T(u_n))
 \end{equation*}
 converges to the unique solution of $\vip(\F,\fix(T))$.
\end{theorem}

Another possibility is that the projection $P_S$ is not known, but $S=\bigcap_{n=1}^NS_n$, where the individual projections $P_{S_n}$ are simple enough to have known closed form expressions \cite{Yamada(01)}. This case is covered by the following Theorem.
\begin{theorem}[Yamada \cite{Yamada(01)}]\label{thm:Yam2}
 For $n=1,\ldots,N$, let $T_n:H\to H$ be nonexpansive mappings that satisfies $S:=\bigcap_{n=1}^N\fix(T_i)\ne\emptyset$, and assume that
 \begin{equation}\label{eq:Bausch}
  F=\fix(T_N\cdots T_1)=\fix(T_1T_N\cdots T_2)=\cdots=\fix(T_{N-1}T_N-2\cdots T_1T_N).\tag{+}
 \end{equation}
Suppose that the mapping $\F:H\to H$ is $\kappa$-Lipschitzian and $\eta$-strongly monotone. Then, for any $u_0\in H$, any $\mu\in(0,2\eta/\kappa^2)$ and any sequence $(\lambda_n)\subset(0,1]$ satisfying
 \begin{center}
 \begin{enumerate*}
   \item $\displaystyle\lim_{n\to\infty}\lambda_n=0$,\quad\phantom{x}
  \item $\displaystyle\sum_{n=1}^{\infty}\lambda_n$ diverges, and\quad\phantom{x}
  \item $\displaystyle\sum_{n=1}^{\infty}\vert\lambda_n-\lambda_{n+N}\vert<\infty$,
 \end{enumerate*}
 \end{center}
 
 \noindent the sequence $(u_n)$ generated by
 \begin{equation*}
  u_{n+1}:=T_{[n+1]}(u_n)-\lambda_{n+1}\mu\F(T_{[n+1]}(u_n))
 \end{equation*}
 converges strongly to the unique solution of $\vip(\F,S)$, where $[n]:=n\mod N$.
\end{theorem}

It should be remarked that Theorem \ref{thm:Yam2} admits $\lambda_n:=1/n$, while Theorem \ref{thm:Yam1} only allows for $\lambda_n:=1/n^\rho$ for $0<\rho<1$. However, since one can choose $N=1$ in Theorem \ref{thm:Yam2}, the choice $\lambda_n:=1/n$ is also covered for the case of a single nonexpansive mapping $T$ with $\fix(T)=S$. Moreover, it is interesting to note that the Bauschke condition \eqref{eq:Bausch} introduced in \cite{Bauschke(96)} is always satisfied whenever $T_n=P_{S_n}$ for closed convex sets $S_n$ with nonempty intersection.

\section{Relation to Moudafi's Viscosity Approximation Method}

Roughly at the same time as Yamada, Moudafi \cite{Moudafi(00)} independently proposed the Viscosity Approximation Method, which is given for nonexpansive $T:C\to C$ and strictly contractive $f:C\to C$ by
\begin{equation*}
 x_{n+1}:=\lambda_{n+1}f(x_n)+(1-\lambda_n)Tx_n.
\end{equation*}
Now observe that Yamada's iteration scheme, the Hybrid Steepest Descent Method, can be rearranged as follows:
\begin{align*}
 u_{n+1}&=Tu_n-\lambda_{n+1}\mu\F(Tu_n)\\
	&=(1-\lambda_{n+1})Tu_n+\lambda_{n+1}(I-\mu\F)(Tu_n).
\end{align*}
Therefore, Yamada's iteration scheme is a special case of the Viscosity Approximation Method if one chooses the contraction $f:=(I-\mu\F)\circ T$. However, Yamada's proof establishing convergence of the HSDM under the proposed conditions is easily reformulated to accomodate for the prima facie more general Viscosity Approximation Method. Moreover, the bounds proposed in this paper also hold for the Viscosity Approximation Method, as the reader may readily verify.

Moreover, both Yamada's and Bruck's conditions imposed on $(\lambda_n)$ do not include the important case $\lambda_n=1/(n+1)$. Xu \cite{Xu(04)} later showed that the Viscosity Approximation Method converges for $\lambda_n=1/(n+1)$ by proving convergence under Wittmann's conditions.
\begin{equation*}
 \text{(i)}\lim\limits_{n\rightarrow\infty}\alpha_n=0,\qquad\text{(ii)}\sum\limits_{n=0}^{\infty}\alpha_n=\infty,\qquad\text{(iii)}\displaystyle\lim_{n\to\infty}(\alpha_n-\alpha_{n-1})/\alpha_n=0
\end{equation*}
However, one should note that Yamada's Theorem for \textit{finitely many mappings} $T_i$ (Theorem \ref{thm:Yam2}) imposes precisely these conditions on $(\lambda_n)$ for the case $N=1$.

Convergence of the Viscosity Approximation Method for finitely many mappings
\begin{equation}\label{eq:final}
 x_{n+1}:=\lambda_{n+1}f(x_n)+(1-\lambda_{n+1})T_{[n+1]}(x_n),\quad [n]:=n\mod N
\end{equation}
was later shown by Jung \cite{Jung(06)}. One again easily verifies that the bounds provided for the Hybrid Steepest Descent Method for the case of a finite family of nonexpansive mappings $T_i$ also holds for the corresponding Viscosity Approximation Method.

Finally, one should observe that the HSDM for a finite family of mappings is, in fact, not a special case of \eqref{eq:final}. In fact, rearranging the HSDM as before, one obtains
\begin{align*}
 u_{n+1}&=T_{[n+1]}(u_n)-\lambda_{n+1}\mu\F(T_{[n+1]}(u_n))\\
	&=(1-\lambda_{n+1})Tu_n+\lambda_{n+1}(I-\mu\F)(T_{[n+1]}u_n)
\end{align*}
Since the contraction $(I-\mu\F)\circ T_{[n+1]}$ now depends on $n$, it is not permitted in the Viscosity Approximation Method.

\section{Rate of convergence versus rate of metastability}
An effective rate of convergence for the iterations of Theorems \ref{thm:Yam1} and \ref{thm:Yam2} to the solution $u^*$ of the VIP is a function $R:(0,\infty)\to\NN$ such that
\begin{equation*}
 \forall\varepsilon>0\forall n\ge R(\varepsilon)\left(\Vert u_n-u^*\Vert<\varepsilon\right).
\end{equation*}
However, effective rates on the strong convergence of $(x_n)$ are generally ruled out. In fact, there are (computable) nonexpansive mappings $f$ on the Hilbert cube (sequences $(x_n)\in\ell_2$ with $\vert x_n\vert\le1$ for all $n$) that have no computable fixed points \cite{Neumann(15)}, and so no sequence approximating any fixed point of $f$ can have a computable rate of convergence. Following general proof-theoretic methods, it is necessary to pass first to an alternate version of Cauchyness, the so-called metastability in the sense of Tao, i.e. (here $[n;n+g(n)]:=\{ n,n+1,n+2,\ldots,n+g(n)\}$) 
\begin{equation*}
  \forall \varepsilon >0\,\forall g:\NN\to\NN\,\exists n\in\NN\,\forall i,
  j\in [n;n+g(n)]\,\big(\Vert u_i-u_j\Vert<\varepsilon\big).
\end{equation*}
Metastability is the so-called Herbrand normal form of (a suitable reformulation of) the Cauchy statement for the sequence $(x_n)$, and, as such, is equivalent to the original statement. This version then becomes finitary in the sense that it only talks about finite subsequences of $(x_n)$ when one additionally has a rate of metastability, which is a bound $\Phi:(0,\infty)\times\NN^\NN\to\NN$ on the existential quantifier:
\begin{equation*}
  \forall \varepsilon >0\,\forall g:\NN\to\NN\,\exists n\le\Phi(\varepsilon,g)\,\forall i,j\in [n;n+g(n)]\,\big(\Vert u_i-u_j\Vert<\varepsilon\big).
\end{equation*}
Such bounds are guaranteed to exist and will be computable on rational accuracies $\varepsilon$ under vastly general conditions on the complexity of the proof \cite{Kohlenbach(08)}. 

Quantitative, finitary versions of all of Theorems \ref{thm:Yam1} and \ref{thm:Yam2}, however, should not only finitize the Cauchyness of $(u_n)$, but also that the strong limit is indeed a solution to the variational inequality problem: For all $\varepsilon\in(0,1]$ and all $g:\NN\to\NN$ there exists an $n\le\Xi(\varepsilon,g)$ and an $\varepsilon'>0$ such that, for all $i,j\in[n,n+g(n)]$ and $v\in\fix(T)$,
 \begin{enumerate}[label=(\roman*)]
  \item\label{it:1} $\Vert u_i-u_j\Vert\le\varepsilon$ for all $i,j\in[n;n+g(n)]$
  \item\label{it:2} $\Vert Tv-v\Vert\le\varepsilon'$ implies $\langle\G u_n-u_n,v-u_n\rangle\le\varepsilon$, where $\G=I-\mu\F$ for suitable $\mu>0$.
 \end{enumerate}
  The new, logically transformed proof of \ref{it:1} and \ref{it:2} is totally elementary in that all ideal principles have been eliminated; one can recover Yamada's original theorem using only the axiom of choice over \textit{quantifier-free} formulas.

\section{A Quantitative Solution to the VIP}

We now examine the structure of the proof of Theorem \ref{thm:Yam1} from a proof-theoretic perspective. Given a nonexpansive mapping $T:H\to H$ and a $\kappa$-Lipschitzian and $\eta$-strongly monotone mapping $\F:H\to H$, fix an arbitrary $\mu\in(0,2\eta/\kappa^2)$. Then, the mapping $T^{(\lambda)}:H\to H$ defined by $T^{(\lambda)}(x):=T(x)-\lambda\mu\F(Tx)$ is a strict contraction for all $\lambda\in(0,1]$. As such, given a sequence $(\lambda_n)\subset(0,1]$, there exists for each nonnegative integer $n$ a unique solution $v_n$ to the equation 
\begin{equation}\label{eq:it}
 v_n=T^{(\lambda_n)}(v_n)=Tv_n-\lambda_n\mu\F(Tv_n).
\end{equation}
Next, Yamada shows using weak sequential compactness that $v_n$ converges weakly to the unique fixed point $u^*$ of $T$ that solves the Variational Inequality Problem $\vip(\F,\fix(T))$. Since, moreover, $\Vert v_n-T(v_n)\Vert$ converges to zero, the demiclosedness principle then implies that the weak limit $u^*$ is a fixed point of $T$. This, in turn, is used to prove using constructive reasoning that $(v_n)$ converges strongly to $u^*$. The final step is then a constructive proof of $\Vert u_n-v_n\Vert\to0$, where $(u_n)$ is the iteration proposed in Theorem \ref{thm:Yam1}.

Structurally, the proof of $v_n\to u^*$ is reminiscent of the proof of the following classical result due to Browder:
\begin{theorem}[\cite{Browder(67a)}]\label{thm:browder}
 Let $H$ be a Hilbert space and $T:H\to H$ be a nonexpansive mapping that maps a bounded, closed and convex subset $C$ of $H$ into itself. Let $v_0$ be an arbitrary point of $C$, and for each $k$ with $0<k<1$, let $U_k(x):=kU(x)+(1-k)v_0$.
 
 Then $U_k$ is a strict contraction of $H$, $U_k$ has a unique fixed point $u_k$ in $C$, and $u_k$ converges as $k\to1$ strongly to a fixed point $u_0$ of $U$ in $C$. The fixed point $u_0$ in $C$ is uniquely specified as the fixed point of $U$ in $C$ closest to $v_0$.
\end{theorem}

The nonconstructive part of the proof of this theorem also consists of weak sequential compactness, and the unique existence of a point that solves a variational inequality. The latter in this case is the variational inequality that characterizes the metric projection $P_{\fix(T)}(v_0)$ of the point $v_0$ onto $\fix(T)$, which reads
\begin{equation*}
 \langle v_0-P_{\fix(T)}(v_0),v-P_{\fix(T)}(v_0)\rangle\le0,\quad\text{for all }v\in\fix(T).
\end{equation*}

An extensive proof-theoretic analysis of this proof has already been carried out by Kohlenbach \cite{Kohlenbach(11)}. By virtue of the complete modularity of the logical machinery employed therein, one can reuse the quantitative versions of the use of weak sequential compactness and the demiclosedness principle.

The unique existence of the solution $u^*$ to the Variational Inequality Problem, on the other hand, is substantially more difficult to constructivize than the existence of the metric projection $P_{\fix(T)}(v_0)$. To make sense of this, observe first that for neither of the two proofs the exact point is needed, but only an $\varepsilon$-approximation. For the projection, this corresponds to finding for all $\varepsilon>0$ a point $u\in C$ such that
\begin{equation}\label{eq:pro}
 Tu=u\wedge\forall v\in C\big(Tv=v\ra\Vert u-v_0\Vert\le\Vert v-v_0\Vert+\varepsilon\big).
\end{equation}
The correct form of a quantitative version, i.e.~the Dialectica interpretation combined with negative translation, of this statement is the one given in the following Lemma:
\begin{lemma}[Lemma 2.6 of \cite{Kohlenbach(11)}]\label{lem:proj}
 Let $v_0\in C$ such that $\diam(C)\le d$. Let $\varepsilon\in(0,1],t\in[0,1],\Delta:C\times(C\to(0,1])\to(0,1]$ and $V:C\times(C\to(0,1])\to C$. Then, one can construct $u:=u_{v_0,T}(t,\varepsilon,\Delta,V)\in C$ and $\varphi:=\varphi_{v_0,T}(t,\varepsilon,\Delta,V):C\to(0,1]$ such that
 \begin{equation*}
  \Vert u-Tu\Vert<\Delta(u,\varphi)
 \end{equation*}
 and
 \begin{align}
  \Vert TV(u,\varphi)-V(u,\varphi)\Vert<\varphi(V(u,\varphi))\ra\Vert v_0-u\Vert^2\le\Vert v_0-V^t(u,\varphi)\Vert^2+\varepsilon,\label{eq:proj_quant}
 \end{align}
 where $V^t(u,\varphi):=(1-t)u+tV(u,\varphi)$. In fact, $u,\varphi$ can be defined explicitly as functionals in $\varepsilon,\Delta$ and $V$ in addition to $v_0$ and $T$.
\end{lemma}
\begin{remark}
 We will sometimes call the functionals $\Delta$ and $V$ \emph{counterfunctions} in the style of the no-counterexample-interpretation due to Kreisel \cite{Kreisel(52a),Kreisel(52b)}.
\end{remark}

Let us now turn to formulating an analogue of this lemma in the context of the VIP. To be able to reuse as much as possible from the previous analysis, it is convenient to reformulate the iteration \eqref{eq:it} as a convex combination: For $\G:=I-\mu\F$, where $\mu\in(0,2\eta/\kappa^2)$, one can re-write \eqref{eq:it} as
\begin{equation*}
  v_n=(1-\lambda_n)Tv_n+\lambda_n\G(Tv_n),
\end{equation*}
and the iteration proposed in Theorem \ref{thm:Yam1} as
\begin{equation*}
 u_{n+1}=(1-\lambda_{n+1})Tu_n+\lambda_{n+1}\G(Tu_n).
\end{equation*}
As remarked earlier, for any choice $\mu\in(0,2\eta/\kappa^2)$, the mapping $\G$ is a strict contraction with Lipschitz constant $\tau:=\sqrt{1-\mu(2\eta-\mu\kappa^2)}<1$. From now on, we simply assume that we are given an \textit{arbitrary} $\tau$-contraction $\G$, making no reference to $\F$.

Now, the operators $T$ and $\G$ need only be defined as self-maps on a closed and convex subset $C$ of $H$. To be able to apply Lemma \ref{lem:proj}, we also assume that $C$ is bounded with $\diam(C)\le d$. This condition, however, is no real restriction, as we will show later on, so that $C=H$ is still admissible for our results (see Corollaries \ref{cor:diam} and \ref{cor:analogo}).

Observe that the characterization stated in Proposition \ref{prop:char}\ref{item:characterisation} of the solution to the $\vip$ is formalized by
\begin{equation*}
 \exists u^*\in C\Big(Tu^*=u^*\wedge\forall v\in C\big(Tv=v\ra\Vert u^*-\G u^*\Vert\le\Vert v-\G u^*\Vert\big)\Big).
\end{equation*}
As already mentioned, we only need the weakened, $\varepsilon$-version of this statement. Analogously to the case of the $\varepsilon$-metric projection \eqref{eq:pro}, this corresponds to
\begin{equation*}
 \forall\varepsilon>0\exists u^*\in C\Big(Tu^*=u^*\wedge\forall v\in C\big(Tv=v\ra\Vert u^*-\G u^*\Vert\le\Vert v-\G u^*\Vert+\varepsilon\big)\Big).
\end{equation*}
The same tools that were used to transform \eqref{eq:pro} now tell us that our task is to solve the following problem:
\begin{prob}\label{prob:proj} Suppose $C$ is a closed, bounded, convex subset of a Hilbert space $H$ with $\diam(C)\le d$ for some nonnegative integer $d$, $T:C\to C$ is nonexpansive and $\G:C\to C$ is $\tau$-contractive. For $\varepsilon\in(0,1],t\in[0,1],\Delta:C\times(C\to(0,1])\to(0,1]$ and $V:C\times(C\to(0,1])\to H$, solve for $u^*$ and $\varphi$ in the formula
 \begin{align}
 &\forall\varepsilon\in(0,1]\forall\Delta:C\times(C\to(0,1])\forall V:C\times(C\to(0,1])\ra C\notag\\
 &\quad\exists u^*\in C\exists\varphi:C\to(0,1]\Big(\Vert Tu^*-u^*\Vert\le\Delta(u^*,\varphi)\,\wedge\label{eq:proj1}\\
 &\qquad\quad\Vert TV(u^*,\varphi)-V(u^*,\varphi)\Vert<\varphi(V(u^*,\varphi))\ra\Vert\G u^*-u^*\Vert^2<\Vert\G u^*-V^t(u^*,\varphi)\Vert^2+\varepsilon\Big),\notag
\end{align}
where, as before, $V^t(u^*,\varphi):=(1-t)u^*+tV(u^*,\varphi)$. 
\end{prob}

By Proposition \ref{prop:char}, the unique point $u^*\in\fix(P_{\fix(T)}\circ\G)$ will solve the $\vip$. The quantitative version of this step is given by the following Lemma.

\begin{lemma}[Lemma 2.7 of \cite{Kohlenbach(11)}]\label{lem:switch}
 Let $u^*,u,v\in H$ such that $\Vert u^*-v\Vert\le d$. For $t\in[0,1]$, define $w_t:=(1-t)u^*+tv$. Then 
 \begin{equation*}
  \forall\varepsilon\in(0,1](\Vert u^*-u\Vert^2\le\frac{\varepsilon^2}{2d^2}+\Vert u-w_{\frac{\varepsilon}{3d^2}}\Vert^2\ra\langle u-u^*,v-u^*\rangle<\varepsilon).
 \end{equation*}
 \end{lemma}

To solve Problem \ref{prob:proj}, recall that by Proposition \ref{prop:char}, $u^*$ is the unique fixed point of the mapping $x\mapsto P_{\fix(T)}(\G x)$. Since the metric projection is nonexpansive and $\G$ is, for proper choice of $\mu$, a strict contraction, this mapping is also a strict contraction. Thus, the Picard iteration, starting with an arbitrary point $p$, converges strongly to $u^*$:
\begin{equation*}
 u^*=\lim_{n\to\infty}(P_{\fix(T)}\circ\G)^{(n)}(p).
\end{equation*}

In view of this, it is not surprising that a quantitative version of the existence of $u^*$ will iterate the solution functionals of Lemma \ref{lem:proj}.

Before we proceed, we need the following variant of Lemma \ref{lem:proj}, as it turns out later that we need to win against two counterfunction pairs $(\Delta_1,V_1)$ and $(\Delta_2,V_2)$, simultaneously:
\begin{lemma}\label{lem:proj_con}
 Let $v_0\in C$ such that $v_0-v\le d$ for some $v\in\fix(T)$. Let $\varepsilon\in(0,1],t_1,t_2\in[0,1],\Delta_1,\Delta_2:C\times(C\to(0,1])\to(0,1]$ and $V_1,V_2:C\times(C\to(0,1])\to C$. Then, one can construct a $u:=u'_{v_0,T}(t_1,t_2,\varepsilon,\Delta_1,\Delta_2,V_1,V_2)\in C$ and a $\varphi:=\varphi'_{v_0,T}(t_1,t_2,\varepsilon,\Delta_1,\Delta_2,V_1,V_2):C\to(0,1]$ such that for $i=1,2$,
 \begin{equation*}
  \Vert u-Tu\Vert<\Delta_i(u,\varphi)
 \end{equation*}
 and
 \begin{align*}
  &\Vert TV_i(u,\varphi)-V_i(u,\varphi)\Vert<\varphi(V_i(u,\varphi))\notag\\
  &\qquad\ra\Vert v_0-u\Vert^2\le\Vert v_0-V_i^{t_i}(u,\varphi)\Vert^2+\varepsilon.
 \end{align*}
\end{lemma}

\begin{proof}
 Given $\Delta_1,\Delta_2:C\times(C\to(0,1])\to(0,1]$ and $V_1,V_2:C\times(C\to(0,1])\to C$, define $\Delta:C\times(C\to(0,1])\to(0,1]$ by
 \begin{equation*}
  \Delta(u,\varphi):=\min\{\Delta_1(u,\varphi),\Delta_2(u,\varphi)\},
 \end{equation*}
 and $V:C\times(C\to(0,1])\to C$ by
 \begin{equation*}
  V(u,\varphi):=\left\{\begin{array}{ll}
                        V_1(u,\varphi),&\textmd{if }\Vert v_0-V_1^{t_1}(u,\varphi)\Vert\le\Vert v_0-V_2^{t_2}(u,\varphi)\Vert,\\
                        V_2(u,\varphi),&\textmd{otherwise.}
                       \end{array}\right.
 \end{equation*}
 Using the solution operators of Lemma \ref{lem:proj}, we define $u'_{v_0,T}(t,\varepsilon,\Delta_1,\Delta_2,V_1,V_2):=u_{v_0,T}(t,\varepsilon,\Delta,V)$ and $\varphi'_{v_0,T}(t,\varepsilon,\Delta_1,\Delta_2,V_1,V_2):=\varphi_{v_0,T}(t,\varepsilon,\Delta,V)$.
\end{proof}

Returning to the original problem, we start with an arbitrary point $p$ in $C$ and use Lemma 2.4 of \cite{Kohlenbach(11)} to obtain a point $u_0$ and a functional $\varphi_0$ which together solve the quantitative version (according to Lemma \ref{lem:proj}) of the $\varepsilon$-projection of $p$ onto $\fix(T)$ for suitable counterfunctions $\Delta_0$ and $V_0$. We then repeat this procedure for $\G u_0$, obtaining a point $u_1$, and so on. In total, we obtain points $u_i$ and functionals $\varphi_i$ such that
\begin{align}
 &\quad\Vert Tu_i-u_i\Vert\le\Delta_i(u_i,\varphi_i)\,\wedge\label{eq:proj2}\\
 &\qquad\quad\Big(\Vert TV_i(u_i,\varphi_i)-V_i(u_i,\varphi_i)\Vert<\varphi_i(V_i(u_i,\varphi_i))\notag\\
 &\qquad\qquad\qquad\qquad\ra\Vert\G u_{i-1}-u_i\Vert^2<\Vert\G u_{i-1}-V_i^t(u_i,\varphi_i)\Vert^2+\varepsilon\Big),\notag
\end{align}
for suitable counterfunctions $\Delta_i$ and $V_i$ which depend on the counterfunctions $\Delta$ and $V$ from statement  \eqref{eq:proj1}. (As before, $V_i^t(u_i,\varphi_i):=(1-t)\G u_i+tV_i(u_i,\varphi_i)$.)

The key in solving Problem \ref{prob:proj} will be the observation that $u_i$ is the $\varepsilon$-projection of $\G u_{i-1}$ with respect to counterfunctions $V_i$ and $\Delta_i$. Therefore, the points $u_i$ are an $\varepsilon$-version of the Picard-iteration of the contractive mapping $P_{\fix(T)}\circ\G$. As such, the distance $\Vert u_i-u_{i-1}\Vert$ can be made arbitrarily small for a sufficiently large $i$, given that we choose our counterfunctions in the correct way -- in this case counterfunctions that ensure that the $\varepsilon$-projection is $\varepsilon$-nonexpansive with respect to the involved points. The simple observation
\begin{align}
 \Vert\G u_i-u_i\Vert^2&\le\Vert\G u_{i-1}-u_i\Vert^2+\tau^2\Vert u_i-u_{i-1}\Vert^2+2\tau\cdot\Vert\G u_{i-1}-u_i\Vert\cdot\Vert u_i-u_{i-1}\Vert\notag\\
	  &\stackrel{\eqref{eq:proj2}}{<}\Vert\G u_{i-1}-V_i^t(u_i,\varphi_i)\Vert^2+\varepsilon+\tau^2\Vert u_i-u_{i-1}\Vert^2+2d\tau\Vert u_i-u_{i-1}\Vert\notag\\
	  &\le\bigl(\Vert\G u_i-V_i^t(u_i,\varphi_i)\Vert+\Vert\G u_i-\G u_{i-1}\Vert\bigr)^2+\varepsilon+\tau^2\Vert u_i-u_{i-1}\Vert^2+2d\tau\Vert u_i-u_{i-1}\Vert\notag\\
	  &\le\Vert\G u_i-V_i^t(u_i,\varphi_i)\Vert^2+2\tau^2\Vert u_i-u_{i-1}\Vert^2+4d\tau\Vert u_i-u_{i-1}\Vert+\varepsilon\label{eq:ineqproj}
\end{align}
then tells us that we may take $u^*=u_i$ if the integer $i$ is large enough to ensure the distance between $u_i$ and $u_{i-1}$ is small enough.

Our task is now to analyze the following proof that the metric projection is nonexpansive. For $x,y\in H$, denote by $Px$ and $Py$ their projections onto an arbitrary convex set. Then
\begin{equation*}
 \langle Px-x,Px-Py\rangle\le0\quad\textmd{and}\quad\langle Py-y,Py-Px\rangle\le0.
\end{equation*}
Summing up these two inequalities yields $\langle Px-Py+y-x,Px-Py\rangle\le0$, which implies
\begin{equation*}
 \Vert Px-Py\Vert\le\langle x-y,Px-Py\rangle\le\Vert x-y\Vert\cdot\Vert Px-Py\Vert,
\end{equation*}
which implies the claim.

Quantitatively, this translates as follows (for later convenience already instantiated with the points $u_i$ and projection onto $\fix(T)$). Suppose that we have for some $\tilde\varepsilon>0$ to be specified later on that
\begin{align}
 &\Vert u_{i+1}-\G u_i\Vert^2\le\frac{\tilde\varepsilon^4}{8d^2}+\left\Vert\left(1-\frac{\tilde\varepsilon^2}{6d^2}\right)u_{i+1}+\frac{\tilde\varepsilon^2}{6d^2}u_i-\G u_i\right\Vert^2\label{eq:eins},\\
 &\Vert u_i-\G u_{i-1}\Vert^2\le\frac{\tilde\varepsilon^4}{8d^2}+\left\Vert\left(1-\frac{\tilde\varepsilon^2}{6d^2}\right)u_i+\frac{\tilde\varepsilon^2}{6d^2}u_{i+1}-\G u_{i-1}\right\Vert^2\label{eq:zwei}.
\end{align}
For notational simplicity later on, we write
\begin{equation*}
 A(\tilde\varepsilon,u,v,p):\equiv \Vert u-p\Vert^2\le\frac{\tilde\varepsilon^4}{8d^2}+\left\Vert\left(1-\frac{\tilde\varepsilon^2}{6d^2}\right)u+\frac{\tilde\varepsilon^2}{6d^2}v-p\right\Vert^2.
\end{equation*}
Then $\eqref{eq:eins}\equiv A(\tilde\varepsilon,u_{i+1},u_i,\G u_i)$ and $\eqref{eq:zwei}\equiv A(\tilde\varepsilon,u_i,u_{i+1},\G u_{i-1})$ for $i\ge 1$. By Lemma \ref{lem:switch}, $A(\tilde\varepsilon,u_{i+1},u_i,\G u_i)$ and $A(\tilde\varepsilon,u_i,u_{i+1},\G u_{i-1})$ together imply
\begin{align*}
 \langle u_{i+1}-\G u_i,u_{i+1}-u_i\rangle<\tilde\varepsilon^2/2,\\
 \langle u_i-\G u_{i-1},u_i-u_{i+1}\rangle<\tilde\varepsilon^2/2.
\end{align*}
Thus, since $\G$ is a $\tau$-contraction,
\begin{align}
 \Vert u_{i+1}-u_i\Vert^2&<\Vert\G u_i-\G u_{i-1}\Vert\cdot\Vert u_{i+1}-u_i\Vert+\tilde\varepsilon^2\notag\\
			 &\le\tau\Vert u_i-u_{i-1}\Vert\cdot\Vert u_{i+1}-u_i\Vert+\tilde\varepsilon^2.\label{eq:drei}
\end{align}
Now, the problem is that, when we divide the inequality by $\Vert u_{i+1}-u_i\Vert$ (if it is strictly greater than 0), the term $\tilde\varepsilon/\Vert u_{i+1}-u_i\Vert$ becomes unbounded for small $\Vert u_{i+1}-u_i\Vert$. However, we want to make $\Vert u_{i+1}-u_i\Vert$ small anyway, so this is not a problem, and it gives rise to the following case distinction.
\begin{enumerate}[label=(\roman*)]
 \item For $\Vert u_{i+1}-u_i\Vert<\tilde\varepsilon$, we immediately get $\Vert u_{i+1}-u_i\Vert<\tilde\varepsilon\le\tau\Vert u_i-u_{i-1}\Vert+\tilde\varepsilon$
 \item For $\Vert u_{i+1}-u_i\Vert\ge\tilde\varepsilon$, we get $\Vert u_{i+1}-u_i\Vert<\tau\Vert u_i-u_{i-1}\Vert+\tilde\varepsilon$ by dividing \eqref{eq:drei} by $\Vert u_{i+1}-u_i\Vert$. 
\end{enumerate}
Thus, we have shown for all integers $i\ge0$ that $A(\tilde\varepsilon,u_{i+1},u_i,\G u_i)$ and $A(\tilde\varepsilon,u_i,u_{i+1},\G u_{i-1})$ imply $\Vert u_{i+1}-u_{i}\Vert<\tau\Vert u_i-u_{i-1}\Vert+\tilde\varepsilon$. Now suppose $\Vert u_{i+1}-u_i\Vert<\tau\Vert u_i-u_{i-1}\Vert+\tilde\varepsilon$ for all nonnegative integers $k\le i$, then
\begin{align*}
 \Vert u_{i+1}-u_i\Vert&<\tau\Vert u_i-u_{i-1}\Vert+\tilde\varepsilon\\
			&<\tau^2\Vert u_{i-1}-u_{i-2}\Vert+\tau\tilde\varepsilon+\tilde\varepsilon\\
			&<\ldots\\
			&<\tau^{i+1}\cdot\Vert u_0-p\Vert+\tilde\varepsilon\cdot\sum_{k=0}^{i-1}\tau^k\\
			&<\tau^{i+1}d+\frac{\tilde\varepsilon}{1-\tau}
\end{align*}

Going back to \eqref{eq:ineqproj} with $\tilde\varepsilon/3$ instead of $\varepsilon$, we see that for $d\ge1$
\begin{align*}
 \Vert\G u_i-u_i\Vert^2
 &<\Vert\G u_i-V_i^t(u_i,\varphi_i)\Vert^2+2\tau^2\Vert u_i-u_{i-1}\Vert^2+4d\tau\Vert u_i-u_{i-1}\Vert+\frac{\tilde\varepsilon}{3}\\
 &<\Vert\G u_i-V_i^t(u_i,\varphi_i)\Vert^2+2d^2\tau^{2i+2}+4d^2\tau^{i+1}+\frac{2\tau^2\tilde\varepsilon^2}{(1-\tau)^2}+\frac{4d\tau^2\tilde\varepsilon}{1-\tau}+\frac{\tilde\varepsilon}{3}\\
 &\le\Vert\G u_i-V_i^t(u_i,\varphi_i)\Vert^2+2d^2\tau^{i+1}(\tau^{i+1}+2)+\frac{2\tilde\varepsilon^2}{(1-\tau)^2}+\frac{4d\tilde\varepsilon}{1-\tau}+\frac{\tilde\varepsilon}{3}\\
 &\le\Vert\G u_i-V_i^t(u_i,\varphi_i)\Vert^2+2d^2\tau^{i+1}(\tau^{i+1}+2)+\frac{3+4d}{(1-\tau)^2}\tilde\varepsilon\\
 &\le\Vert\G u_i-V_i^t(u_i,\varphi_i)\Vert^2+3d^2\tau^{i+1}+\frac{3+4d}{(1-\tau)^2}\tilde\varepsilon\\
 &\le\Vert\G u_i-V_i^t(u_i,\varphi_i)\Vert^2+\varepsilon,
\end{align*}
for $i\ge i_0:=\lceil\log_{\tau}(\tilde\varepsilon/6d^2)-1\rceil$ and $\tilde\varepsilon:=\frac{(1-\tau)^2}{6+8d}\varepsilon$. In other words
\begin{multline}\label{eq:impl}
 \Bigg\{A(\tilde\varepsilon,u_{i_0},u_{i_0-1},\G u_{i_0-1})\wedge A(\tilde\varepsilon,u_0,u_1,p)\wedge\bigand{k=1}{i_0-1})\Bigl(\tilde\varepsilon,A(u_i,u_{i+1},\G u_{i-1})\wedge A(\tilde\varepsilon,u_i,u_{i-1},\G u_{i-1}\Bigr)\\\hspace{-4cm}\wedge\Big(\Vert TV(u_{i_0},\varphi_{i_0})-V(u_{i_0},\varphi_{i_0})<\varphi_{i_0}(V(u_{i_0},\varphi_{i_0}))\\\ra\Vert u_{i_0}-\G u_{i_0-1}\Vert^2\le\frac{\tilde\varepsilon}{3}+\left\Vert(1-t)u_{i+1}+tV(u_{i_0},\varphi_{i_0})-\G u_{i_0-1}\right\Vert^2\Big)\Bigg\}\\\ra\Bigg\{\bigg(\Vert TV(u_{i_0},\varphi_{i_0})-V(u_{i_0},\varphi_{i_0})\Vert<\varphi_{i_0}(V(u_{i_0},\varphi_{i_0}))\\\ra\Vert\G u_{i_0}-u_{i_0}\Vert^2\le\Vert\G u_{i_0}-V_{i_0}^t(u_{i_0},\varphi_{i_0 })\Vert^2+\varepsilon\bigg)\Bigg\}.
\end{multline}
Therefore, we need to construct the finite sequence $(u_i)_{0\le i\le i_0}$ satisfying \eqref{eq:eins} and \eqref{eq:zwei}. Then $u_{i_0}$ and $\varphi_{i_0}$ will solve Problem \ref{prob:proj}.

From these considerations, it is clear that for $i=0$, we will need to win against a convex combination (of known weight and accuracy, see \eqref{eq:eins} and \eqref{eq:zwei}) of ourselves, \ie$u_0$, and the subsequent point $u_1$, the $\varepsilon$-projection of $\G u_0$, which we anticipate as the outcome of the iterative process with reference point $\G u_0$ given by Lemma \ref{lem:proj}. Namely, we choose as counterfunction the anticipated next point. But this anticipated next point needs to win both against convex combinations of its predecessor (the point we are trying to construct right now!) and its successor, cf.~\eqref{eq:zwei}.

\begin{notation}
 Suppose that $t(x)$ is a mathematical expression that depends on a variable $x$. Then $\lambda x.t(x)$ denotes the function mapping $x$ to $t(x)$. For example, $\lambda n.n$ for integers $n$ denotes the identity on the integers. Likewise, $\lambda x.x^2$ for real numbers $x$ denotes the square-function on the reals. This notation will prove highly convenient in the sequel.
\end{notation}

For $0\le i\le i_0$, the considerations mentioned above give rise to the counterfunctions (using the previously introduced notation)

\begin{align*}
 V_i(u,\varphi):=\left\{\begin{array}{ll}
	 V(u,\varphi),&\textmd{for }i=i_0,\\
	 u'_{\G u,T}(t,\tilde\varepsilon^2/6d^2,\tilde\varepsilon^4/8d^2,\Delta_{i+1},\lambda v\lambda\psi.\varphi(v),V_{i+1},\lambda v\lambda\psi.u),\quad&\textmd{for }i=i_0-1,\\
         U(\Delta_{i+1},\lambda v\lambda\psi.\varphi(v),V_{i+1},\lambda v\lambda\psi.u,\G u),&\textmd{for }i\le i_0-2,\\
        \end{array}\right.
\end{align*}
and
\begin{align*}
 \Delta_i(u,\varphi):=\left\{\begin{array}{ll}
                              \Delta(u,\varphi),&\textmd{for }i=i_0,\\
			      \varphi'_{\G u,T}(t,\tilde\varepsilon^2/6d^2,\tilde\varepsilon^4/8d^2,\Delta_{i+1},\lambda v\lambda\psi.\varphi(v),V_{i+1},\lambda v\lambda\psi.u)(u),&\textmd{for }i=i_0-1,\\
			      \Phi(\Delta_{i+1},\lambda v\lambda\psi.\varphi(v),V_{i+1},\lambda v\lambda\psi.u,\G u)(u),&\textmd{for }i\le i_0-2,\\
                             \end{array}\right.
\end{align*}
where
\begin{align*}
 U(\Delta,\Delta',V,V',u)&:=u'_{\G u,T}(\tilde\varepsilon^2/6d^2,\tilde\varepsilon^2/6d^2,\tilde\varepsilon^4/8d^2,\Delta,\Delta',V,V'),\\
 \Phi(\Delta,\Delta',V,V',u)&:=\varphi'_{\G u,T}(\tilde\varepsilon^2/6d^2,\tilde\varepsilon^2/6d^2,\tilde\varepsilon^4/8d^2,\Delta,\Delta',V,V'),
\end{align*}
and $u',\varphi'$ are the solution functionals of Lemma \ref{lem:proj_con}. Moreover, $V$ and $\Delta$ are the original counterfunctions of Problem \ref{prob:proj}, i.e.~of the original problem.
Now set
\begin{align*}
 u_i:=\left\{\begin{array}{ll}
                              u_{p,T}(\tilde\varepsilon^2/6d^2,\tilde\varepsilon^4/8d^2,\Delta_0,V_0),&\textmd{for }i=0,\\
			     U(\Delta_i,\lambda v\lambda\psi.\varphi_{i-1}(v),V_i,\lambda v\lambda\psi x.u_{i-1},\G u_{i-1}),&\textmd{for }1\le i\le i_0-1,\\
			      u'_{\G u_{i-1},T}(t,\tilde\varepsilon^2/6d^2,\tilde\varepsilon^4/8d^2,\Delta_i,\lambda v\lambda\psi.\varphi_{i-1}(v),V_i,\lambda v\lambda\psi.u_{i-1}),\quad&\textmd{for }i=i_0,\\
                             \end{array}\right.
\end{align*}
and
\begin{align*}
 \varphi_i:=\left\{\begin{array}{ll}
                              \varphi_{p,T}(\tilde\varepsilon^2/6d^2,\tilde\varepsilon^4/8d^2,\Delta_0,V_0),&\textmd{for }i=0,\\
			      \Phi(\Delta_i,\lambda v\lambda\psi.\varphi_{i-1}(v),V_i,\lambda v\lambda\psi.u_{i-1},\G u_{i-1}),&\textmd{for }1\le i\le i_0-1,\\
			      \varphi'_{\G u_{i-1},T}(t,\tilde\varepsilon^2/6d^2,\tilde\varepsilon^4/8d^2,\Delta_i,\lambda v\lambda\psi.\varphi_{i-1}(v),V_i,\lambda v\lambda\psi.u_{i-1}),\quad&\textmd{for }i=i_0,\\
                             \end{array}\right.
\end{align*}
for some arbitrary point $p\in C$. We now show that these counterfunctions are as required.

For $0\le i\le i_0-2$, the points $u_i$ and the functions $\varphi_i$ satisfy by Lemma \ref{lem:proj_con}
\begin{align*}
 \Vert u_i-Tu_i\Vert&<\Delta_i(u_i,\varphi_i)\\
	  &=\Phi(\Delta_{i+1},\lambda v\lambda\psi.\varphi_i(v),V_{i+1},\lambda v\lambda\psi.u_i,\G u_i)(u_i)\\
	  &=\varphi_{i+1}(u_i).
\end{align*}
Similarly,
\begin{align*}
 \Vert u_{i_0-1}-&Tu_{i_0-1}\Vert<\Delta_{i_0-1}(u_{i_0-1},\varphi_{i_0-1})\\
				&=\varphi'_{\G u_{i_0-1},T}(t,\varepsilon^2/24d^2,\varepsilon^4/32d^2,\Delta_{i_0},\lambda v\lambda\psi.\varphi_{i_{0-1}}(v),V_{i_0},\lambda v\lambda\psi.u_{i_0-1})(u_{i_0-1})\\
				&=\varphi_{i_0}(u_{i_0-1}).
\end{align*}
Consequently, $\Vert u_i-Tu_i\Vert<\varphi_{i+1}(u_i)$ for all $0\le i\le i_0-1$. Moreover, for $1\le i\le i_0$,
\begin{align*}
 \Vert u_i-Tu_i\Vert&<(\lambda v\lambda\psi.\varphi_{i-1}(v))(u_i,\varphi_i)\\
		    &=\varphi_{i-1}(u_i).
\end{align*}
Furthermore, $\Delta_{i_0}=\Delta$, so $\Vert u_{i_0}-Tu_{i_0}\Vert<\Delta(u_{i_0},\varphi_{i_0})$.

Now recall that, for notational simplicity, we denoted formula \eqref{eq:eins} by the formula $A(\tilde{\varepsilon},u_i,u_{i+1},\G u_{i-1})$. The second part of Lemma \ref{lem:proj_con} then reads for $1\le i\le i_0-1$
\begin{align*}
 &\Vert TV_i(u_i,\varphi_i)-V_i(u_i,\varphi_i)\Vert<\varphi_i(V_i(u_i,\varphi_i))\ra A(\tilde{\varepsilon},u_i,V_i(u_i,\varphi_i),\G u_{i-1}),\textmd{ and}\\
 &\Vert Tu_{i-1}-u_{i-1}\Vert<\varphi_i\bigl((\lambda v\lambda\psi.u_{i-1})(u_i,\varphi_i)\bigr)\ra A(\tilde{\varepsilon},u_i,(\lambda v\lambda\psi.u_{i-1})(u_i,\varphi_i),\G u_{i-1}).
\end{align*}
But observe that $V_i(u_i,\varphi_i)=U(\Delta_{i+1},\lambda v\lambda\psi.\varphi_i(v),V_{i+1},\lambda v\lambda\psi.u_i,\G u_i)=u_{i+1}$ and, regarding the second implication, $(\lambda v\lambda\psi.u_{i-1})(u_i,\varphi_i)=u_{i-1}$. Thus, the above implications read 
\begin{align*}
 &\Vert Tu_{i+1}-u_{i+1}\Vert<\varphi_i(u_{i+1})\ra A(\tilde{\varepsilon},u_i,u_{i+1},\G u_{i-1}),\textmd{ and}\\
 &\Vert Tu_{i-1}-u_{i-1}\Vert<\varphi_i(u_{i-1})\ra A(\tilde{\varepsilon},u_i,u_{i-1},\G u_{i-1}),\textmd{ for }1\le i\le i_0-1.
\end{align*}
Since the $V_{i_0}=V$ and $V_0(u_0,\varphi_0)=u_1$, we also get
\begin{multline*}
 \Vert TV(u_{i_0},\varphi_{i_0})-V(u_{i_0},\varphi_{i_0})\Vert<\varphi_{i_0}(V(u_{i_0},\varphi_{i_0}))\\\ra\Vert u_{i_0}-\G u_{i_0-1}\Vert^2\le\frac{\varepsilon^4}{32d^2}+\left\Vert(1-t)u_{i+1}+tV(u_{i_0},\varphi_{i_0})-\G u_{i_0-1}\right\Vert^2
\end{multline*}
and $\Vert Tu_1-u_1\Vert<\varphi_0(u_1)\ra A(\tilde{\varepsilon},u_0,u_1,p)$. Applying the modus ponens and using \eqref{eq:impl}, we then see that $u_{i_0}$ and $\varphi_{i_0}$ are, in fact, solutions for Problem \ref{prob:proj}.

\section{Majorizing the Solution Functionals}
Following [5, 6], we define a notion of majorization for the functionals involved in our solution to Problem \ref{prob:proj}.
\begin{definition}\label{def:maj2}
 \begin{enumerate}[label=(\roman*)]
  \item We say that a function $\varphi:C\to(0,1]$ is majorized by $k\in\NN^*$ if $1/k\le\varphi(v)$ for all $v\in C$. In this case, we write $k\maj\varphi$.
  \item We say that a function $\Delta:C\times(C\to(0,1])\to(0,1]$ is majorized by $f:\NN^*\to\NN^*$ if, for all $\varphi:C\to(0,1]$ and $k\in\NN^*$,
    \begin{equation*}
      k\maj\varphi\ra f(k)\maj\lambda v.\Delta(v,\varphi).
    \end{equation*}
  \item We say that the solution operator $\Phi$ of Problem \ref{prob:proj} and Lemma \ref{lem:proj} (suppressing dependence on the parameters $\varepsilon$ and $t$) is majorized by $\Phi^*:(\NN^*\to\NN^*)\to\NN^*$ if, for all $V$, $\Delta$ and $f$ as before,
  \begin{equation*}
   f\maj\Delta\ra\Phi^*(f)\maj\Phi(\Delta,V).
  \end{equation*}
  \item Similarly, the solution operator $\Phi'$ of Lemma \ref{lem:proj_con} is majorized by $\Phi'^*$ if, for all $V_1,V_2$, $\Delta_1,\Delta_2$ and $f_1,f_2$ as before,
  \begin{equation*}
   f_1\maj\Delta_1\wedge f_2\maj\Delta_2\ra\Phi'^*(f_1,f_2)\maj\Phi(\Delta_1,\Delta_2,V_1,V_2).
  \end{equation*}
 \end{enumerate}
\end{definition}
We now show how to majorize the solution operator $\Psi$ of Problem \ref{prob:proj}. To do so, we first need to majorize the solution operator of Lemma \ref{lem:proj}, which can be stated explicitly as follows \cite{Kohlenbach(11)}: For $i\le n_{\varepsilon}:=\lceil d^2/\varepsilon\rceil$ we define $\psi_i:C\to(0,1]$ and $u_i\in C$ inductively by
 \begin{align*}
  \psi_1(\Delta,V)&:=\lambda v.1&u_1(\Delta,V)&:=\hat u\in \fix(T)\\
  \psi_{i+1}(\Delta,V)&:=\lambda v.\Delta'(v,\psi_i(\Delta,V))&u_{i+1}(\Delta,V)&:=V'(u_i(\Delta,V),\psi_{n_{\varepsilon}-i-1}(\Delta,V)),
 \end{align*}
where
\begin{align*}
 \Delta'(u,\psi)&:=\min\{\Delta(u,\psi^u),\psi^u(V(u,\psi^u))\},\\
 V'(u,\psi)	&:=(1-t)u+t(V(u,\psi^u)),\\
 \psi^u(v)	&:=\frac{\psi((1-t)u+tv)^2}{16d}.
\end{align*}
Then, for some $i\le n_{\varepsilon}$, we have that $u_i(\Delta,V)$ and $\psi^{u_i}_{n_\varepsilon-i}(\Delta,V)$ satisfy the claim. We write $u_{v_0,T}(t,\varepsilon,\Delta,V):=u_i$ and $\varphi_{v_0,T}(t,\varepsilon,\Delta,V):=\psi^{u_i}_{n_\varepsilon-i}$, where $i$ is the least index such that $u_i,\psi^{u_i}_{n_\varepsilon-i}$ satisfy the claim of Lemma \ref{lem:proj}.

\begin{notation}
 Given any function $f:\NN^*\to\NN^*$, define the function $f^M:\NN^*\to\NN^*$ by $f^M(n):=\max\{f(i):1\le n\}$. Observe that $f^M$ is monotone for any given $f$.
\end{notation}

\begin{lemma}\label{lem:maj_psi}
\begin{enumerate}[label=(\roman*)]
 \item\label{lem:maja}The functional $\varphi^*:{\NN^*}^{\NN^*}\to\NN^*$ defined by $\varphi^*(f):=\max\{\psi_i^*(f):1\le i\le n_{\varepsilon}\}$, where $\psi^*_i:{\NN^*}^{\NN^*}\to\NN^*$ is defined recursively by
 \begin{align*}
   \psi_1^*(f)&:=1,\\
   \psi_{i+1}^*(f)&:=\max\{f(16d\cdot\psi_i^*(f)^2),16d\cdot\psi_i^*(f)^2\},
 \end{align*}
is a majorant to the solution operator $\varphi$ of Lemma \ref{lem:proj}, i.e.~$\varphi^*\maj\varphi$.
 \item\label{lem:maj_psi_spec} The functional $\tilde\varphi^*:{\NN^*}^{\NN^*}\to\NN^*$ defined by $\tilde\varphi^*(f):=\tilde f^{(n_{\tilde\varepsilon})}(1)$ is also a majorant to the solution operator $\varphi$ of Lemma \ref{lem:proj}, where $\tilde f:\NN^*\to\NN^*$ is defined by $\tilde f(n):=\max\{f^M(16dn^2),16dn^2\}$.
 \item\label{lem:maj_psi_spec1} If $f$ is a nondecreasing function, then $\tilde\varphi^*(f)=\varphi^*(f)$, and $f^M=f$.
 \item Given any majorant $\varphi^*$ to the solution operator of Lemma \ref{lem:proj}, the function $\varphi^+:{\NN^*}^{\NN^*}\times{\NN^*}^{\NN^*}\to\NN^*$ defined by $\varphi^+(f_1,f_2):=\psi^*(\max\{f_1,f_2\})$ majorizes the solution operator $\varphi'$ of Lemma \ref{lem:proj_con}. Here $\max\{f_1,f_2\}$ denotes the pointwise maximum of the two functions $f_1$ and $f_2$.
\end{enumerate}
\end{lemma}
\begin{proof}
\begin{enumerate}[label=(\roman*)]
 \item
 We first show that $\psi_i^*\maj\psi_i$ by induction on $i$. For $i=1$, the claim is trivial since $\psi_0(\Delta,V)(v)=1$ for all $v\in C$. Now, suppose that $\psi_i^*\maj\psi_i$ for some positive integer $i$ and $f\maj\Delta$. Then,
\begin{enumerate}
 \item on the one hand, using the induction hypothesis and the definition of majorization, $\psi_i^*(f)\maj\psi_i(\Delta,V)$. The observation
 \begin{equation*}
  k\maj\psi\ra 16dk^2\maj\psi^u,\quad\text{for all }u\in C
 \end{equation*}
 then implies $16d\cdot \psi_i^*(f)\maj\psi_i^u(\Delta,V)$.
 \item On the other hand, $f\maj\Delta$ by definition implies 
 \begin{equation*}
  k\maj\psi^u_i(\Delta,V)\ra f(k)\maj\lambda v.\Delta(v,\psi^u_i(\Delta,V))
 \end{equation*}
 But the induction hypothesis implies as before $16d\cdot\psi_i^*(f)\maj\psi_i^u(\Delta,V)$, so $f(16d\cdot \psi_i^*(f))\maj\lambda v.\Delta(v,\psi_i^u(\Delta,V))$.
\end{enumerate}
 In total, $\psi_{i+1}^*\maj\psi_{i+1}$. That $\varphi^*$ is a common majorant for all $\psi_i$, where $i\le n_{\varepsilon}$, follows from Lemma 6.4 of \cite{Kohlenbach(08)}. Therefore $\varphi^*\maj\varphi$.
 \item First, we show by induction on $i$ that $\tilde f^{(i)}(1)\ge\psi_i^*(f)$. For $i=1$, the statement holds with equality. Moreover,
 \begin{align*}
  \psi_{i+1}^*(f)&=\max\{f(16d\cdot\psi_i^*(f)^2),16d\cdot\psi^*_i(f)^2\}\\
	  &\le\max\{f^M(16d\cdot\psi_i^*(f)^2),16d\cdot\psi^*_i(f)^2\}	  
 \end{align*}
 By the monotonicity of $f^M$ and the induction hypothesis, we conclude
 \begin{align*}
 \psi_{i+1}^*(f)&\le\max\left\{f^M\left(16d\cdot(\tilde f^{(i)}(1))^2\right),16d\cdot\left(\tilde f^{(i)}(1)\right)^2\right\}\\
		&=\tilde f\left(\tilde f^{(i)}(1)\right)\\
		&=\tilde f^{(i+1)}(1).
 \end{align*}
 Therefore, $\tilde f^{(i)}(1)\ge\psi_i^*(f)$ for all $i$. Since $\tilde f$ is monotone, $\tilde f^{(n_{\varepsilon})}(1)\ge\tilde f^{(i)}(1)$ for all $i\le n_{\varepsilon}$, so the claim follows from part \ref{lem:maja}.
 \item $\varphi^*(f)=\tilde\varphi^*(f)$ for nondecreasing $f$ is shown as in the previous part with equality throughout. 
 \item Suppose $f_i\maj\Delta_i$ for $i=1,2$. Then $\max\{f_1,f_2\}\maj\Delta_i$ as well, so we conclude that $\max\{f_1,f_2\}\maj\lambda u\lambda\varphi.\min\{\Delta_1(u,\varphi),\Delta_2(u,\varphi)\}$. Consequently, since $\varphi^*\maj\varphi$ by hypothesis, we obtain $\varphi^+\maj\varphi'$.
\end{enumerate}
\end{proof}

\begin{lemma}\label{lem:maj_delt}
 Given a majorant $f\maj\Delta$, define a function $f_i:\NN^*\to\NN^*$ by $f_i(k):=\tilde f^{\left(n_{\tilde\varepsilon}^i\right)}(k)$, where $\tilde f$ is defined as in Lemma \ref{lem:maj_psi} and $n_{\tilde\varepsilon}:=\lceil8d^4/\varepsilon^4\rceil$. Then $f_i\maj\Delta_i$ for $0\le i\le i_0$.
\end{lemma}

\begin{proof}
We show by (backward) induction on $n$ that for any majorant $\varphi^*$ of $\varphi$, the functions $\hat f_i$
\begin{equation}\label{eq:intermed}
	 \left\{\begin{array}{l}
	 \hat f_{i_0}:=\tilde f\maj\Delta_{i_0},\text{ and}\\
         \hspace{4pt}\hat f_i:=\lambda k.\varphi^*(\max\{f_{i+1},\lambda n.k\})\maj\Delta_i,\textmd{ for }i\le i_0-1,
	 \end{array}\right.
\end{equation}
majorize $\Delta_i$, respectively. By definition, $f_{i_0}=f\maj\Delta=\Delta_{i_0}$, completing the induction base. Now recall that, by definition, $f_i\maj\Delta_i$ if and only $k\maj\varphi\ra f_i(k)\maj\lambda v.\Delta_i(v,\varphi)$ for $0\le i\le i_0-1$. So suppose $f_{i+1}\maj\Delta_{i+1}$ and $k\maj\varphi$. Then, $\lambda n.k\maj\lambda\psi.\varphi$. Thus, the induction hypothesis $f_{i+1}\maj\Delta_{i+1}$ implies using the last part of Lemma \ref{lem:maj_psi}
\begin{equation*}
 \lambda k.\varphi^*(\max\{f_{i+1},\lambda n.k\})\maj\Delta_i,
\end{equation*}
Completing the proof of \eqref{eq:intermed}.

We now prove by induction on $i$ that $\hat f_i(k)\le f_i(k)$ for all $i$ and $k$, which will complete the proof of the lemma. The induction start $i=0$ is trivial. For notational simplicity, we write $g_{i,k}(n):=\max\{\hat f_i(n),k\}$. Now observe that, since $\tilde f$ is monotone and satisfies $f(n)\ge n$ for all positive integers $n$, so does $\hat f_i$ for each $i$. Therefore, parts \ref{lem:maj_psi_spec} and \ref{lem:maj_psi_spec1} of Lemma \ref{lem:maj_psi} imply
\begin{align*}
 \hat f_{i_0-i-1}(k)&=g_{i_0-i,k}^{(n_{\tilde\varepsilon})}(1)=g_{i_0-i,k}^{(n_{\tilde\varepsilon}-1)}(g_{i_0-i,k}(1))=g_{i_0-i,k}^{(n_{\tilde\varepsilon}-1)}(\max\{\hat f_{i_0-i}(1),k\})\\
  &=g_{i_0-i,k}^{(n_{\tilde\varepsilon}-1)}(\max\{\hat f_{i_0-i}(1),k\})=\hat f_{i_0-i}^{(n_{\tilde\varepsilon}-1)}(\max\{\hat f_{i_0-i}(1),k\}).
\end{align*}
Using the induction hypothesis and the monotonicity of $\tilde f$, we then see that
\begin{equation*}
 \hat f_{i_0-i-1}(k)=\max\left\{\hat f_{i_0-i}^{(n_{\tilde\varepsilon})}(1),\hat f_{i_0-i}^{(n_{\tilde\varepsilon}-1)}(k)\right\}\le\hat f_{i_0-i}^{(n_{\tilde\varepsilon})}(k)\le \tilde f_{i_0-i}^{(n_{\tilde\varepsilon})}(k)=\tilde f^{\left(n_{\tilde\varepsilon}^i\cdot n_{\tilde\varepsilon}\right)}(k)=f_{i_0-i-1}(k).
\end{equation*}
\end{proof}

 \begin{lemma}\label{lem:maj_delta}
  Suppose $f:\NN^*\to\NN^*$ is monotone, satisfies $f(n)\ge n$ for all positive integers $n$ and $f\maj\Delta$. 
 \end{lemma}
 
 \begin{proof}
  Define for each nonnegative integer $i\le i_0$ the integer $k_i$ by
  \begin{equation*}
   k_0:=\tilde f_0^{(n_{\tilde\varepsilon})}(1),\quad k_{i+1}:=\tilde f_i^{(n_{\tilde\varepsilon})}(k_i).
  \end{equation*}
  We first show that $k_i\maj\varphi_i$ for all $0\le i\le i_0$ by induction on $i$. The base case follows from \ref{lem:maj_psi_spec1} of Lemma \ref{lem:maj_psi} using the fact that $f_i$ is monotone. The induction step follows from Lemma \ref{lem:maj_psi} and (we write $g_{i,k}(n):=\max\{\tilde f_i(n),\lambda n.k_{i-1}\}$)
  \begin{equation*}
   g_{i,k}^{(n_{\tilde\varepsilon})}(1)=g_{i,k}^{(n_{\tilde\varepsilon}-1)}(\max\{\tilde f_i(1),k_{i-1}\})=g_{i,k}^{(n_{\tilde\varepsilon}-1)}(k_{i-1})=\tilde f_i^{(n_{\tilde\varepsilon}-1)}(k_{i-1})\le\tilde f_i^{(n_{\tilde\varepsilon})}(k_{i-1})=k_i.
  \end{equation*}
 \end{proof}
 
 We can now state the solution to Problem \ref{prob:proj}:
 
 \begin{theorem}\label{thm:proj}
   Suppose $C$ is a closed, bounded, convex subset of a Hilbert space $H$ with $\diam(C)\le d$ for some nonnegative integer $d$, $T:C\to C$ is nonexpansive and $\G:C\to C$ is $\tau$-contractive. For $\varepsilon\in(0,1],t\in[0,1],\Delta:C\times(C\to(0,1])\to(0,1]$ and $V:C\times(C\to(0,1])\to H$, one can construct $u:=u_{v_0,T}(t,\varepsilon,\Delta,V)\in C$ and $\varphi:=\varphi_{v_0,T,t,\varepsilon}(\Delta,V):C\to(0,1]$ such that
 \begin{equation*}
  \Vert u-Tu\Vert<\Delta(u,\varphi)
 \end{equation*}
 and
 \begin{align*}
  &\Vert TV(u,\varphi)-V(u,\varphi)\Vert<\varphi(V(u,\varphi))\\
  &\qquad\ra\Vert\G u-u\Vert^2<\Vert(1-t)\G u-tV(u,\varphi)\Vert^2+\varepsilon
 \end{align*}
 In fact, $u,\varphi$ can be defined explicitly as functionals in $\Delta,V$. Moreover, if we define a mapping $K:{\NN^*}^{\NN^*}\to\NN^*$ by $K(f):=k_{i_0}(\tilde f)$, then
 \begin{equation*}
  K\maj\varphi,
 \end{equation*}
 where $f_i(k):=\tilde f^{\left(n_{\tilde\varepsilon}^i\right)}(k)$ and
 \begin{align*}
 k_0(\tilde f)&:=\tilde f_0^{(n_{\tilde\varepsilon})}(1),&k_{i+1}(f)&:=\tilde f_i^{(n_{\tilde\varepsilon})}(k_i(f))\\
 \tilde f(n)&:=\max\{f^M(16dn^2),16dn^2\},&n_{\tilde\varepsilon}&:=\lceil 8d^4/\tilde\varepsilon^4\rceil,\\
 i_0&:=\lceil\log_{\tau}(\tilde\varepsilon/6d^2)-1\rceil,&\tilde\varepsilon&:=\frac{(1-\tau)^2}{6+8d}\varepsilon.
 \end{align*}
 \end{theorem}

\section{Strong Convergence Results}

In this section, we prove our main results for the case of a single nonexpansive mapping $T:C\to C$. We start by giving a quantitative version of the convergence of the resolvent $(v_n)$, where $v_n$ is defined for each nonnegative integer $n$ as the unique point satisfying the equation
\begin{equation}\label{eq:hybridmethod}
 v_n=(1-\lambda_n)Tv_n+\lambda_n\G Tv_n.
\end{equation}
and $(\lambda_n)\subset(0,1]$ is a null sequence.

% Before we proceed, we first correct an error in Lemma 3.1(a) of \cite{Yamada(01)}, which states that the mapping $T^{(\lambda)}:C\to C$ defined by $T^{(\lambda)}(x):=(1-\lambda)Tx+\lambda\G(Tx)$ is a strict contraction with Lipschitz constant $1-\lambda\tau$. While $T^{(\lambda)}$ is indeed a strict contraction, the Lipschitz constant computed in the proof of the Lemma is incorrectly stated as $(1-\tau)$, when it is actually $(1-\lambda(1-\tau))$. This can be seen by correcting the simple arithmetic error preceding (12) of \cite{Yamada(01)}.
\begin{lemma}[cf.~\cite{Yamada(01)}]\label{lem:contractive}
  The mapping $T^{(\lambda)}:C\to C$ defined by $T^{(\lambda)}(x):=(1-\lambda)Tx+\lambda\G(Tx)$ is a strict contraction with Lipschitz constant $(1-\lambda(1-\tau))$. 
 \end{lemma}
%  The incorrect Lipschitz constant for $T^{(\lambda)}$ is used throughout the paper. Therefore, whenever we cite an inequality of \cite{Yamada(01)} that involves the incorrect constant, we tacitly use it with the correct one, i.e.~$1-\lambda(1-\tau)$ in place of $1-\lambda\tau$.
%  
 First of all, we need the following lemma, which is similar to \cite{Kohlenbach(11)}:
 \begin{lemma}
 Suppose $\lambda_n\in(0,1]$, $u^*,v\in C$, $d\in\NN$ and $h:\NN\to\NN$ satisfy $\lambda_n\ge \frac{1}{h(n)}$ and $\Vert v_n-u^*\Vert\le d$. Then
 
 \vspace{10pt}
 \begin{enumerate*}[label=(\roman*)]
  \item\label{item:core1} $\displaystyle\Vert Tu^*-u^*\Vert\le\frac{\varepsilon^2}{9d(1-\tau)\cdot h(n)},\quad$
  \item\label{item:core2} $\displaystyle\left\langle\G u^*-u^*,v_n-v\right\rangle\le\frac{\varepsilon^2}{3(1-\tau)},\text{ and}\quad$\newline
  \item\label{item:core3} $\displaystyle\langle\G u^*-u^*,v-u^*\rangle\le\frac{\varepsilon^2}{3(1-\tau)}$
 \end{enumerate*}
 
 \vspace{10pt}
 \noindent imply $\Vert v_n-u^*\Vert\le\varepsilon$.

\end{lemma}

\begin{proof}
 Observe that
 \begin{align}\label{eq:core1}
  &(1-\lambda_n)(v_n-u^*-Tv_n+Tu^*)+\lambda_n(v_n-u^*-\G Tv_n+\G Tu^*)\notag\\
  &\quad=v_n-u^*-Tv_n+Tu^*-\lambda_n(v_n-u^*)+\lambda_nTv_n-\lambda_nTu^*+\lambda_n(v_n-u^*)\notag\\
  &\quad\qquad-\lambda_n\G Tv_n+\lambda_n\G Tu^*\notag\\
  &\quad=v_n-Tv_n+\lambda_nTv_n-\lambda_n\G Tv_n-u^*+Tu^*-\lambda_nTu^*+\lambda_n\G Tu^*\notag\\
  &\quad=Tu^*-u^*+\lambda_n(\G Tu^*-Tu^*).
 \end{align}
 Moreover,
 \begin{align}\label{eq:core2}
  \left\langle v_n-u^*-Tv_n+Tu^*,v_n-u^*\right\rangle&=\Vert v_n-u^*\Vert^2-\left\langle Tv_n-Tu^*,v_n-u^*\right\rangle\notag\\
						     &\ge\Vert v_n-u^*\Vert^2-\Vert Tv_n-Tu^*\Vert\cdot\Vert v_n-u^*\Vert\notag\\
						     &\ge\Vert v_n-u^*\Vert^2-\Vert v_n-u^*\Vert^2=0,
 \end{align}
 and
 \begin{align}\label{eq:core3}
  &\lambda_n\left\langle v_n-u^*-\G Tv_n+\G Tu^*,v_n-u^*\right\rangle\notag\\
  &\quad=\lambda_n\Vert v_n-u^*\Vert^2-\lambda_n\left\langle\G Tv_n-\G Tu^*,v_n-u^*\right\rangle\notag\\
  &\quad\ge\lambda_n\Vert v_n-u^*\Vert^2-\lambda_n\Vert\G Tv_n-\G Tu^*\Vert\cdot\Vert v_n-u^*\Vert\notag\\
  &\quad\ge\lambda_n(1-\tau)\Vert v_n-u^*\Vert^2.
 \end{align}
 Combining \eqref{eq:core1}, \eqref{eq:core2} and \eqref{eq:core3},
 \begin{align*}
  \lambda_n(1-\tau)\Vert v_n-u^*\Vert^2&\le\lambda_n\left\langle v_n-u^*-\G Tv_n+\G Tu^*,v_n-u^*\right\rangle\\
      &\le\lambda_n\left\langle v_n-u^*-\G Tv_n+\G Tu^*,v_n-u^*\right\rangle\\
      &\qquad+(1-\lambda_n)\left\langle v_n-u^*-Tv_n+Tu^*,v_n-u^*\right\rangle\\
      &\le\left\langle Tu^*-u^*,v_n-u^*\right\rangle+\lambda_n\left\langle\G Tu^*-Tu^*,v_n-u^*\right\rangle\\
      &\le\Vert Tu^*-u^*\Vert\cdot\Vert v_n-u^*\Vert+\lambda_n\left\langle\G Tu^*-Tu^*,v_n-u^*\right\rangle\\
      &=d\cdot\Vert Tu^*-u^*\Vert+\lambda_n\left\langle\G Tu^*-\G u^*,v_n-u^*\right\rangle\\
      &\qquad+\lambda_n\left\langle\G u^*-u^*,v_n-u^*\right\rangle+\lambda_n\left\langle u^*-Tu^*,v_n-u^*\right\rangle\\
      &\le d(1+\tau\lambda_n+\lambda_n)\cdot\Vert Tu^*-u^*\Vert+\lambda_n\langle\G u^*-u^*,v_n-u^*\rangle\\
      &=d(1+\tau\lambda_n+\lambda_n)\cdot\Vert Tu^*-u^*\Vert+\lambda_n\left\langle\G u^*-u^*,v_n-v\right\rangle\\
      &\qquad+\lambda_n\langle\G u^*-u^*,v-u^*\rangle.
 \end{align*}
Therefore,
\begin{equation*}
 (1-\tau)\Vert v_n-u^*\Vert^2\le\frac{3d}{\lambda_n}\cdot\Vert Tu^*-u^*\Vert+\left\langle\G u^*-u^*,v_n-v\right\rangle+\langle\G u^*-u^*,v-u^*\rangle.
\end{equation*}
Now, the claim follows from the assumptions \ref{item:core1}, \ref{item:core2} and \ref{item:core3}.
\end{proof}

\begin{corollary}\label{cor:core}
 If we instantiate $v:=v_n$, then \ref{item:core2} becomes true with `$\,=0$' instead of `$\,\le\varepsilon^2/3(1-\tau)$', so we get that
 \begin{equation*}
  \Vert Tu^*-u^*\Vert\le\frac{\varepsilon^2}{6d(1-\tau)\cdot h(n)},\quad\text{and}\quad\left\langle\G u^*-u^*,v_n-u^*\right\rangle\le\frac{\varepsilon^2}{2(1-\tau)}
 \end{equation*}
imply $\Vert v_n-u^*\Vert\le\varepsilon$.
\end{corollary}

From here on, we follow except for a few minor details the argumentation of \cite{Kohlenbach(11)}. For the sake of completeness, we adapt the proof to our situation.

\begin{lemma}
 For $t\in(0,1]$, denote by $z_t$ the unique point satisfying $z_t=(1-t)Tz_t+t\G Tz_t$. Then $\Vert z_t-Tz_t\Vert<\varepsilon$ for all $\varepsilon>0$ and $0<t<\varepsilon/d$.
\end{lemma}
\begin{proof}
 Follows from
 \begin{equation*}
  \Vert z_t-Tz_t\Vert=t\Vert Tz_t-\G Tz_t\Vert\le td<\varepsilon.
 \end{equation*}
\end{proof}

\begin{lemma}[Lemma 2.9 of \cite{Kohlenbach(11)}]\label{lem:subseq}
 Let $X$ be a normed linear space. Then the following holds:
 \begin{equation*}
  \forall\varepsilon>0\,\forall g:\NN\to\NN\,\forall u\in X\forall(v_n)\subset X\,\forall m\in\NN\left(\Vert v_{g_{u,\varepsilon}(m)}-u\Vert\le\varepsilon/2\ra\Vert v_{g(m)}-v_m\Vert\le\varepsilon\right),
 \end{equation*}
 where
 \begin{equation*}
  g_{u,\varepsilon}(m):=\left\{\begin{array}{ll}
                                g(m),&\text{if }\Vert v_{g(m)}-u\Vert>\varepsilon/2,\\
                                m,&\text{otherwise.}
                               \end{array}\right.
 \end{equation*}
\end{lemma}

\begin{lemma}[Lemma 2.13 of \cite{Kohlenbach(11)}]\label{lem:sub}
 Let $\chi:\NN\to\NN$ be a rate of convergence of $(\lambda_n)$ towards $0$, i.e.~$\lambda_i\le\frac{1}{n+1}$ for all nonnegative integers $n$ and all $i\ge\chi(n)$. Then, for $(v_n)$ as defined in \eqref{eq:hybridmethod} and $\tilde g_{u,\varepsilon}$ defined as in Lemma \ref{lem:subseq} (but with $\tilde g(n):=\max\{n,g(n)\}$),
 \begin{multline*}
  \forall\varepsilon\in(0,1]\,\forall g:\NN^*\to\NN^*\,\forall\varphi:C\to(0,1]\,\forall u\in C\,\forall k\maj\varphi\\
	  \left(\Vert Tv_{\tilde g_{u,\varepsilon}(\chi(d\cdot k))}-v_{\tilde g_{u,\varepsilon}(\chi(d\cdot k))}\Vert<\varphi(v_{\tilde g_{u,\varepsilon}(\chi(d\cdot k))})\right).
 \end{multline*}
\end{lemma}

\begin{theorem}\label{thm:main}
 Let $H$ be a real Hilbert space, $d\in\NN^*$ and $C\subset H$ be a bounded closed convex subset with $d\ge\diam C$. Let $T:C\to C$ be nonexpansive and $\G:C\to C$ be a strict contraction with Lipschitz constant $\tau<1$. Let $(\lambda_n)$ be a sequence in $(0,1]$ that converges towards $0$ and $h:\NN\to\NN^*$ such that $\lambda_n\ge\frac{1}{h(n)}$ for all $n\in\NN$. Let $\chi:\NN\to\NN$ be a rate of convergence of $(\lambda_n)$ towards $0$, i.e.~$\lambda_i\le\frac{1}{n+1}$ for all nonnegative integers $n$ and all $i\ge\chi(n)$. Denote by $v_n$ the unique solution to the equation
 \begin{equation*}
  v_n=(1-\lambda_n)Tv_n+\lambda_n\G Tv_n.
 \end{equation*}
 Then, for all $\varepsilon\in(0,1]$ and $g:\NN\to\NN^*$
 \begin{equation*}
  \exists j\le\Xi(\varepsilon,g,\chi,h,d)\left(\Vert v_j-v_{\tilde g(j)}\Vert\le\varepsilon\right),
 \end{equation*}
  where $\tilde g(n):=\max\{n,g(n)\}$ and
  \begin{equation*}
   \Xi(\varepsilon,g,\chi,h,d):=\chi\left(d\cdot k_{i_0}(\tilde f)\right),
  \end{equation*}
  where $i_0:=\lceil\log_{\tau}(\tilde\varepsilon/6d^2)-1\rceil$ and
  \begin{align*}
   f(n)&:=\left\lceil\frac{6d(1-\tau)h^M(\tilde g^M(\chi(d\cdot n)))}{(\varepsilon/2)^2}\right\rceil,\\
   \tilde f(n)&:=\max\{f^M(16dn^2),16dn^2\},&\tilde g(n)&:=\max\{n,g(n)\},\\
 k_0(\tilde f)&:=\tilde f_0^{(n_{\tilde\varepsilon})}(1),&k_{i+1}(\tilde f)&:=\tilde f_i^{(n_{\tilde\varepsilon})}(k_i(f)),\\
  f_i(k)&:=\tilde f^{\left(n_{\tilde\varepsilon}^i\right)}(k),&\tilde f_i(n)&:=\max\{f_i^M(16dn^2),16dn^2\},\\
 n_{\tilde\varepsilon}&:=\lceil 8d^4/\tilde\varepsilon^4\rceil,&\tilde\varepsilon&:=\frac{(1-\tau)^2}{6+8d}\varepsilon_d,\\
  \varepsilon_d&:=\frac{(\varepsilon/2)^4}{8(1-\tau)^2d^2}.
 \end{align*}
\end{theorem}

\begin{proof}
 For $\varepsilon\in(0,1]$ and $g:\NN\to\NN$ define analogously to \cite{Kohlenbach(11)} a functional $J_{\varepsilon,g}:C\times(C\to(0,1])\to\NN$ by
 \begin{equation*}
  J_{\varepsilon,g}(u,\varphi):=\left\{\begin{array}{ll}
                                        \min\left\{j\in\NN:\Vert T(v_{\tilde g_{u,\varepsilon}(j)})-v_{\tilde g_{u,\varepsilon}(j)}\Vert<\varphi(v_{\tilde g_{u,\varepsilon}(j)})\right\},&\text{if such a $j$ exists,}\\
                                        0,&\text{otherwise,}
                                       \end{array}\right.
 \end{equation*}
 where $\tilde g_{u,\varepsilon}$ is defined as $g_{u,\varepsilon}$ in Lemma \ref{lem:subseq} with $\tilde g$ instead of $g$.
 Observe that by Lemma \ref{lem:sub}, we are always in the first case of the definition of $J_{\varepsilon,g}$ whenever $\varphi:C\to(0,1]$ is majorizable. Moreover, we have
 \begin{equation}\label{eq:j}
  k\maj\varphi\ra\forall u\in C\bigl(J_{\varepsilon,g}(u,\varphi)\le\chi(d\cdot k)\bigr).
 \end{equation}
 Now define $V_{\varepsilon,g}:C\times(C\to(0,1])\to C$ by $V_{\varepsilon,g}(u,\varphi):=v_{\tilde g_{u,\varepsilon}(J_{\varepsilon,g}(u,\varphi))}$, $\varepsilon_d:=\frac{(\varepsilon/2)^4}{8(1-\tau)^2d^2}$ and $t:=\frac{(\varepsilon/2)^2}{6(1-\tau)^2d^2}$. Moreover, if we define $\Delta_{\varepsilon,g}(u,\varphi):=\frac{(\varepsilon/2)^2}{6d(1-\tau)\cdot h(\tilde g_{u,\varepsilon}(J_{\varepsilon,g}(u,\varphi)))}$, then, given a majorant $k\maj\varphi$, \eqref{eq:j} implies $J_{\varepsilon,g}(u,\varphi)\le\chi(d\cdot k)$ for all $u\in C$. Hence
 \begin{equation}\label{eq:fork}
  k\maj\varphi\ra\forall u\in C\left(\frac{(\varepsilon/2)^2}{6d(1-\tau)\cdot h(\tilde g_{u,\varepsilon}(J_{\varepsilon,g}(u,\varphi)))}\ge\left\lceil\frac{{(\varepsilon/2)^2}}{6d(1-\tau)h^M\left(\tilde g^M\left(\chi(d\cdot k)\right)\right)}\right\rceil\right).
 \end{equation}
 Therefore, $f\maj\Delta_{\varepsilon,g}$. We write $\tilde u:=U(\varepsilon_d,t,\Delta_{\varepsilon,g},V_{\varepsilon,g})$ and $\tilde\varphi:=\Phi(\varepsilon_d,t,\Delta_{\varepsilon,g},V_{\varepsilon,g})$ to simplify notation, where $\Phi$ and $U$ are the solution functionals to Problem \ref{prob:proj}. By Theorem \ref{thm:proj}, we then get $k_{i_0}(f)\maj\tilde\varphi$, whence \eqref{eq:j} implies $J_{\varepsilon,g}(\tilde u,\tilde\varphi)\le\Xi(\varepsilon,g,\chi,h,d)$.

 Then, for $j:=J_{\varepsilon,g}(\tilde u,\tilde\varphi)$ and $v:=V_{\varepsilon,g}(\tilde u,\tilde\varphi)$
 \begin{equation}\label{eq:fix}
  \Vert T\tilde u-\tilde u\Vert<\frac{(\varepsilon/2)^2}{6d(1-\tau)\cdot h(\tilde g_{\tilde u,\varepsilon}(j))}
 \end{equation}
  and
 \begin{equation*}
  \Vert Tv-v\Vert<\tilde\varphi(v)\ra\Vert\G\tilde u-\tilde u\Vert^2<\Vert\G\tilde u-(1-t)\tilde u-tv\Vert^2+\varepsilon_d.
 \end{equation*}
  But  $\Vert Tv-v\Vert<\tilde\varphi(v)$ by construction of $J_{\varepsilon,g}$, so 
  \begin{equation}\label{eq:fix2}
    \Vert\G\tilde u-\tilde u\Vert^2<\Vert\G\tilde u-(1-t)\tilde u-tv\Vert^2+\varepsilon_d.
  \end{equation}
  Lemma \ref{lem:switch} then yields $\left\langle\G\tilde u-\tilde u,v-\tilde u\right\rangle<\frac{(\varepsilon/2)^2}{(1-\tau)}$, so Corollary \ref{cor:core} implies 
  \begin{equation}\label{eq:vtildeu}
   \Vert v-\tilde u\Vert\le\varepsilon/2.
  \end{equation}
  From Lemma \ref{lem:subseq}, and the definitions of $v$ and $J_{\varepsilon,g}$, we conclude $\Vert v_{\tilde g(j)}-v_j\Vert\le\varepsilon$. 
\end{proof}

\begin{corollary}\label{cor:main}
 For all $\varepsilon\in(0,1]$ and $g:\NN\to\NN$, there exists an $n\le\Xi(\varepsilon/2,\lambda n.n+g(n),\chi,h,d)$ such that
 \begin{equation*}
  \Vert v_i-v_j\Vert\le\varepsilon,\quad\text{for all }i,j\in[n;n+g(n)].
 \end{equation*}
\end{corollary}

\begin{proof}
 Follows as in \cite{Kohlenbach(11)}.
\end{proof}

\begin{lemma}[Modulus of Continuity for the VIP]\label{lem:mod}
 Suppose $u,v,w\in C$ satisfy $\Vert u-v\Vert\le\frac{\varepsilon}{2d(2+\tau)}$ and $\langle\G u-u,w-u\rangle\le\varepsilon/2$. Then $\langle\G v-v,w-v\rangle\le\varepsilon$.
\end{lemma}

\begin{proof}
Follows from
 \begin{align*}
  \langle\G v-v,w-v\rangle&=\langle\G v-v,w-u\rangle+\langle\G v-v,u-v\rangle\\
			  &=\langle\G u-u,w-u\rangle+\langle\G v-\G u,w-u\rangle+\langle u-v,w-u\rangle+\langle\G v-v,u-v\rangle\\
			  &\le\frac{\varepsilon}2+d(2+\tau)\Vert u-v\Vert\le\varepsilon.
 \end{align*}
\end{proof}

\begin{theorem}\label{thm:main2}
 In the situation of Theorem \ref{thm:main}, suppose that $\phi_1,\phi_2:(0,\infty)\to\NN$ satisfy
 \begin{enumerate}[label=(\roman*)]
  \item\label{item:div} $\sum_{i=1}^{\phi_1(k)}\lambda_i\ge k\quad$ for all nonnegative integers $k$, and
  \item $\frac{\vert\lambda_n-\lambda_{n+1}\vert}{\lambda_{n+1}^2}\le\varepsilon\quad$ for all $\varepsilon>0$  and $n\ge\phi_2(\varepsilon).$
 \end{enumerate}
 Define the sequence $(u_n)$ by $u_{n+1}:=(1-\lambda_{n+1})Tu_n+\lambda_{n+1}\G Tu_n$ for an arbitrary starting point $u_0\in C$. Then, for all $\varepsilon\in(0,1]$, $g:\NN\to\NN$ and $v\in C$, there exists an $n\le\Xi(\varepsilon/6,g_c,\chi,h,d)+c$, where $g_c(n):=n+c+g(n+c)$ and $c:=\phi_1\left(\frac{(\phi_2((1-\tau)\varepsilon/6d)+\log(6d/\varepsilon))}{1-\tau}\right)$ such that
\begin{equation*}
 \Vert u_i-u_j\Vert\le\varepsilon\quad\text{for all }i,j\in[n;n+g(n)].
\end{equation*}

\end{theorem}
\begin{proof}
If we define $p:=\phi_2((1-\tau)\varepsilon/6d)$, Equation (29) of \cite{Yamada(01)} implies (see also the remarks preceding Lemma \ref{lem:contractive})
 \begin{equation}\label{eq:simpl}
  \Vert u_n-v_n\Vert\le\Vert u_p-v_p\Vert\prod_{i=p+1}^n(1-\lambda_i(1-\tau))+\varepsilon/6,\text{ for all }n\ge p.
 \end{equation}
 Now, for $n\ge\phi_1\left(\frac{(p+\log(6d/\varepsilon))}{1-\tau}\right)$, we have $\sum_{i=1}^n\lambda_i(1-\tau)\ge p+\log(6d/\varepsilon)$, so
 \begin{equation*}
  \sum_{i=p+1}^n\lambda_i(1-\tau)\ge-p(1-\tau)+\sum_{i=1}^n\lambda_i(1-\tau)\ge\log(6d/\varepsilon).
 \end{equation*}
 Therefore, since $-x>\log(1-x)$ for all $0<x<1$,
 \begin{align*}
    \frac{\varepsilon}{6d}\ge\exp\left(-\sum_{i=p+1}^n\lambda_i(1-\tau)\right)&>\exp\left(\sum_{i=p+1}^n\log(1-\lambda_i(1-\tau))\right)\\&=\exp\left(\log\left(\prod_{i=p+1}^n(1-\lambda_i(1-\tau))\right)\right).
 \end{align*}
Now observe that $\phi_1(k)\ge k$ for all nonnegative integers $k$, so $\phi_1\left(\frac{(p+\log(6d/\varepsilon))}{1-\tau}\right)\ge p$. Going back to \eqref{eq:simpl}, we therefore see that 
\begin{equation}\label{eq:ref1}
 \Vert u_n-v_n\Vert\le\varepsilon/3,\text{ for all }n\ge\phi_1\left(\frac{(\phi_2((1-\tau)\varepsilon/6d)+\log(6d/\varepsilon))}{1-\tau}\right).
\end{equation}
Moreover, by Corollary \ref{cor:main}, there exists an $n\le\Xi(\varepsilon/6,\tilde g,\chi,h,d)$ such that 
\begin{equation*}
 \Vert v_i-v_j\Vert\le\varepsilon/3,\quad\text{for all }i,j\in[n,n+c+g(n+c)].
\end{equation*}
Thus, $\tilde n:=n+c$ satisfies
\begin{equation*}
 \Vert u_i-u_j\Vert\le\Vert u_i-v_i\Vert+\Vert u_j-v_j\Vert+\Vert v_i-v_j\Vert\le\varepsilon,\quad\text{for all }i,j\in[\tilde n,\tilde n+g(\tilde n)].
\end{equation*}

\end{proof}

\begin{theorem}\label{thm:mainquant}
In the situation of Theorem \ref{thm:main2}, the following holds: For all $\varepsilon\in(0,1]$, $g:\NN\to\NN$ and $x\in C$, there exists an $n\le\Xi(\delta/6,g_c,\chi,h,d)+c$, where $\delta=\frac{\varepsilon}{2d(2+\tau)}$, $g_c(n):=n+c+g(n+c)$ and $c:=\phi_1\left(\frac{1}{\tau}(\phi_2(\delta/6)+\log(6d/\delta))\right)$ such that
 \begin{enumerate}[label=(\roman*)]
  \item $\Vert u_i-u_j\Vert\le\varepsilon$ for all $i,j\in[n;n+g(n)]$, and
  \item if $x\in C$ satisfies $\Vert Tx-x\Vert\le\varepsilon'$, then $\langle\G u_n-u_n,x-u_n\rangle\le\varepsilon$,
 \end{enumerate}
 where $\varepsilon':=k_{i_0}'(f)$ and $k_{i_0}'(f)$ is defined as $k_{i_0}(f)$ in Theorem \ref{thm:main}, but with $\delta/6$ instead of $\varepsilon$.
\end{theorem}

\begin{proof}
 We first prove that for all $\varepsilon\in(0,1]$, $x\in C$ and $g:\NN\to\NN$, there exists a nonnegative integer $j\le\Xi(\varepsilon,g,\chi,h,d)$ and a $\tilde u'\in C$ such that
 \begin{equation}\label{eq:lilien}
  \Vert v_j-v_{\tilde g(j)}\Vert\le\varepsilon,\text{ and}\quad\Vert Tx-x\Vert\le\varepsilon'\ra\langle\G\tilde u'-\tilde u',x-\tilde u'\rangle\le\frac{(\varepsilon/2)^2}2,
 \end{equation}
 where $\varepsilon':=\alpha^{(i_0\cdot(n_\varepsilon-1))}(\max\{\varphi^*(f),\alpha(1)\})$.

In the proof of Theorem \ref{thm:main}, after equation \eqref{eq:fork}, one can alter the counterfunction $V_{\varepsilon,g}$ to $V_{\varepsilon,g}':C\times(C\to(0,1])\to C$ defined by
 \begin{equation*}
  V_{\varepsilon,g}'(u,\varphi):=\left\{\begin{array}{ll}
                        V_{\varepsilon,g}(u,\varphi),\quad\textmd{if }\Vert\G u-V_{\varepsilon,g}^t(u,\varphi)\Vert\le\left\Vert\G u-(1-t)u-tx\right\Vert,\\
                        x,\quad\textmd{otherwise.}
                       \end{array}\right.
 \end{equation*}
 For $\tilde u':=U(\varepsilon_d,t,\Delta_{\varepsilon,g},V'_{\varepsilon,g})$ and $\tilde\varphi':=\Phi(\varepsilon_d,t,\Delta_{\varepsilon,g},V_{\varepsilon,g}')$ we then get for $j:=J_{\varepsilon,g}(\tilde u',\tilde\varphi')$ and $v:=V'_{\varepsilon,g}(\tilde u',\tilde\varphi')$ as before
 \begin{equation}\label{eq:ana1}
    \Vert T\tilde u'-\tilde u'\Vert<\frac{(\varepsilon/2)^2}{6d\cdot h(\tilde g_{\tilde u,\varepsilon}(j))}
 \end{equation}
   and
 \begin{equation*}
  \Vert Tv-v\Vert<\tilde\varphi(v)\ra\Vert\G\tilde u'-\tilde u'\Vert^2<\Vert\G\tilde u'-(1-t)\tilde u'-tv\Vert^2+\varepsilon_d.
 \end{equation*}
  Then, by construction of $V'_{\varepsilon,g}(u,\varphi)$, we now get two implications; As before,
  \begin{multline}\label{eq:ana2}
   \Vert TV_{\varepsilon,g}(\tilde u',\varphi')-V_{\varepsilon,g}(\tilde u',\varphi')\Vert<\tilde\varphi(V_{\varepsilon,g}(\tilde u',\varphi'))\\\ra\Vert\G\tilde u'-\tilde u'\Vert^2<\Vert\G\tilde u'-(1-t)\tilde u'-tV_{\varepsilon,g}(\tilde u',\varphi')\Vert^2+\varepsilon_d,
  \end{multline}
  and, additionally, 
\begin{equation}\label{eq:new}
 \Vert Tx-x\Vert<\tilde\varphi'(x)
 \ra\Vert\G\tilde u'-\tilde u'\Vert^2<\Vert\G\tilde u'-(1-t)\tilde u'-tx\Vert^2+\varepsilon_d.
\end{equation}
Now, \eqref{eq:ana1} and \eqref{eq:ana2} imply as before $\Vert v_j-v_{\tilde g(j)}\Vert\le\varepsilon$. Moreover,
\begin{equation*}
 \Vert Tx-x\Vert<\tilde\varphi'(x)\ra\langle\G\tilde u'-\tilde u',x-\tilde u'\rangle<\frac{(\varepsilon/2)^2}{2}.
\end{equation*}
 Now observe that the majorant of the solution operator $\Phi$ is independent of the counterfunction $V_{\varepsilon,g}'$; therefore, we may take the same majorant for $\tilde\varphi'$ as we took for $\tilde\varphi$. Therefore
\begin{equation}\label{eq:coinc}
 \Vert Tx-x\Vert<k_{i_0}(f)\ra\langle\G\tilde u'-\tilde u',x-\tilde u'\rangle<\frac{(\varepsilon/2)^2}{2}.
\end{equation}
This completes the proof of \eqref{eq:lilien}.

Thus, we get in Theorems \ref{thm:main} and \ref{thm:main2} also the additional conclusion \eqref{eq:coinc}. Thus, as before in Theorem \ref{thm:main2}, we get an $n\le\Xi(\delta/6,g_c,\chi,h,d)+c$ such that
\begin{equation*}
 \Vert u_i-u_j\Vert\le\delta\le\varepsilon,\quad\text{for all }i,j\in[n;n+g(n)].
\end{equation*}
Moreover, as in the situation of Theorem \ref{thm:main}, we get $\Vert v-\tilde u'\Vert\le\frac{\delta/6}2$ (compare \eqref{eq:vtildeu}). Similarly, we get as in Theorem \ref{thm:main2} that $\Vert u_n-v_n\Vert\le\delta/3$ (compare \eqref{eq:ref1}). Observe also that $\tilde g_{\tilde u,\varepsilon}(n)$ is either $n$ or $\tilde g(n)$, so $v=v_{\tilde g_{\tilde u,\varepsilon}(n)}$ implies
\begin{align}\label{eq:lil}
 \Vert u_n-\tilde u'\Vert&\le\Vert u_n-v_n\Vert+\Vert v_n-v\Vert+\Vert v-\tilde u'\Vert\notag\\
	&\le\frac{\delta}{3}+\Vert v_n-v_{\tilde g(n)}\Vert+\frac{\delta}{12}\le\delta=\frac{\varepsilon}{2d(2+\tau)}.
\end{align}
The claim follows from \eqref{eq:lil} and \eqref{eq:coinc} using Lemma \ref{lem:mod}.
\end{proof}

\begin{corollary}\label{cor:diam}
 For all of the above results, one can drop the condition of $C$ being bounded with $\diam(C)\le d$ in favor of $T$ having a fixed point $v$ such that $\Vert u_0-v\Vert\le d/2$, $\Vert v-\G v\Vert\le\frac{d(1-\tau)}{4}$ and $\Vert v-w\Vert\le\frac{d}{4(1+\tau)}$, where $w$ is the unique fixed point of $\G$.
\end{corollary}
\begin{proof}
 By Lemma \ref{lem:contractive}, we have for all nonnegative integers $n$
 \begin{align*}
  \Vert u_{n+1}-v\Vert&\le\Vert T^{(\lambda_{n+1})}(u_n)-T^{(\lambda_{n+1})}(v)\Vert+\Vert T^{(\lambda_{n+1})}(v)-v\Vert\\
		      &\le(1-\lambda_{n+1}(1-\tau))\Vert u_n-v\Vert+\lambda_{n+1}(1-\tau)\cdot\frac{\Vert \G v-v\Vert}{1-\tau}.
 \end{align*}
 Since $\Vert u_0-v\Vert\le d/2$, we conclude by induction that $\Vert u_n-v\Vert\le d/2$ for all nonnegative integers $n$.
 
 Moreover, observe that Lemma \ref{lem:contractive} implies
 \begin{align*}
  \Vert v_{n}-v\Vert&\le\Vert T^{(\lambda_{n})}(v_n)-T^{(\lambda_{n})}(v)\Vert+\Vert T^{(\lambda_{n})}(v)-v\Vert\\
		&\le(1-\lambda_n(1-\tau))\Vert v_n-v\Vert+\lambda_n(1-\tau)\cdot\frac{\Vert v-\G v\Vert}{1-\tau},
 \end{align*}
 so, since $\lambda_n$ is strictly positive, $\Vert v_n-v\Vert\le\frac{\Vert\G v-v\Vert}{1-\tau}\le d/4\le d/2$. Moreover,
\begin{equation*}
 \Vert \G v_n-v\Vert\le\Vert\G v_n-\G w\Vert+\Vert w-v\Vert\le\tau\Vert v_n-w\Vert+\Vert v-w\Vert\le\tau\Vert v_n-v\Vert+(1+\tau)\Vert v-w\Vert.
\end{equation*}
Therefore, the sequences $(v_n)$, $(\G v_n)$ and $(u_n)$ remain in the ball of radius $d/2$ (and therefore diameter $d$) around $v$. Since the estimate $\diam(C)\le d$ was only ever used for elements of the sequences $(v_n)$, $(\G v_n)$ and $(u_n)$, and convex combinations of those elements, the claim follows.

\end{proof}

\section{Finite Families}

For the rest of this section, let $C$ be a closed and convex subset of $H$. Suppose that $T_1,\ldots,T_N:C\to C$ are nonexpansive mappings with a common fixed point $p\in C$ which satisfy $\bigcap_{i=1}^N\fix(T_i)=\fix(T_N\cdots T_1)$. Then a function $\hat\rho:\NN\times(0,\infty)\to(0,\infty)$ is a modulus for this property if, for all nonnegative integers $d$, all $x\in C$ and all $\varepsilon>0$
\begin{equation}\label{eq:rho}
 \Vert x-p\Vert\le d\text{ and }\Vert T_N\cdots T_1x-x\Vert<\hat\rho(d,\varepsilon)\text{ imply }\Vert T_ix-x\Vert<\varepsilon,\text{ for all }1\le i\le N.
\end{equation}
It is clear that one can, without loss of generality, assume that $\hat\rho$ is monotone in $\varepsilon$ and satisfies $\hat\rho(d,\varepsilon)\le\varepsilon$ for all $\varepsilon>0$ and all $d\in\NN$, which we do from now on.

In \cite{Yamada(01)}, Yamada actually assumes that 
\begin{equation*}
 \bigcap_{i=1}^N\fix(T_i)=\fix(T_N\cdots T_1)=\fix(T_{N-1}\cdots T_1T_N)=\ldots=\fix(T_1\cdots T_N),
\end{equation*}
which is the well-known Bauschke condition \cite{Bauschke(96)}. In \cite{Suzuki(06)}, however, Suzuki showed$^1$\footnote{$^1$ The author is most greatful to Prof. Genaro López Acedo for pointing out this result.} that the Bauschke condition is already implied by the case for e.g. $T_N\cdots T_1$. We now give a quantitative account of this:

\begin{theorem}\label{thm:perm}
 Supppose $C$ is a bounded closed convex subset of a Hilbert space $H$ with diameter $\diam(C)\le d$, and the nonexpansive mappings $T_1,\ldots,T_N$ satisfy \eqref{eq:rho}. Then, if
 \begin{equation*}
  \Vert T_{N-k}\cdots T_1T_N\cdots T_{N-k+1}x-x\Vert<\hat\rho\left(d,\frac{\varepsilon}{2N+1}\right)
 \end{equation*}
 holds for some $k\in\{1,\ldots,N-1\}$, then
 \begin{equation*}
  \Vert T_ix-x\Vert<\varepsilon,\quad\text{for all }i\in\{1,\ldots,N\}.
 \end{equation*}

\end{theorem}

\begin{proof}
 Suppose $\Vert T_{N-k}\cdots T_1T_N\cdots T_{N-k+1}x-x\Vert<\hat\rho(d,\placeholder)$, where we write $\placeholder:=\varepsilon/(2N+1)$. Then, since $T_{N}\cdots T_{N-k+1}$ is nonexpansive,
\begin{equation*}
 \Vert T_N\cdots T_1T_N\cdots T_{N-k+1}x-T_N\cdots T_{N-k+1}x\Vert<\hat\rho(d,\placeholder)
\end{equation*}
By hypothesis \eqref{eq:rho}, this implies
\begin{equation}\label{eq:ultimate}
 \Vert T_iT_N\cdots T_{N-k+1}x-T_N\cdots T_{N-k+1}x\Vert<\placeholder,\quad\text{for all }i\in\{1,\ldots,N\}.
\end{equation}
Therefore,
\begin{align*}
 \Vert T_N\cdots T_{N-k+1}x-x\Vert&\le\Vert T_{N-k}\cdots T_1T_N\cdots T_{N-k+1}x-x\Vert\\
 &\quad+\Vert T_{N-k}\cdots T_1T_N\cdots T_{N-k+1}x-T_N\cdots T_{N-k+1}x\Vert\\
 &<\hat\rho(d,\placeholder)+\Vert T_{N-k}\cdots T_2T_N\cdots T_{N-k+1}x-T_N\cdots T_{N-k+1}x\Vert\\
 &\quad+\Vert T_{N-k}\cdots T_2T_1T_N\cdots T_{N-k+1}x-T_{N-k}\cdots T_2T_N\cdots T_{N-k+1}x\Vert\\
 &\le\hat\rho(d,\placeholder)+\Vert T_{N-k}\cdots T_2T_N\cdots T_{N-k+1}x-T_N\cdots T_{N-k+1}x\Vert\\
 &\quad+\Vert T_1T_N\cdots T_{N-k+1}x-T_N\cdots T_{N-k+1}x\Vert\\
 &\le\hat\rho(d,\placeholder)+\placeholder+\Vert T_{N-k}\cdots T_2T_N\cdots T_{N-k+1}x-T_N\cdots T_{N-k+1}x\Vert\\
 &\le\ldots\\
 &\le\hat\rho(d,\placeholder)+(N-k-1)\placeholder+\Vert T_{N-k}T_N\cdots T_{N-k+1}x-T_N\cdots T_{N-k+1}x\Vert\\
 &\le\hat\rho(d,\placeholder)+(N-k)\placeholder\\
 &\le\hat\rho(d,\placeholder)+(N-1)\placeholder\le N\placeholder.
\end{align*}
Using once more \eqref{eq:ultimate}, we then get
\begin{align*}
 \Vert T_ix-x\Vert&\le\Vert T_ix-T_iT_N\cdots T_{N-k+1}x\Vert+\Vert T_iT_N\cdots T_{N-k+1}x-T_N\cdots T_{N-k+1}x\Vert\\
     &\qquad+\Vert T_N\cdots T_{N-k+1}x-x\Vert\\
     &<(2N+1)\placeholder=\varepsilon,\quad\text{for all }i\in\{1,\ldots,N\}.
\end{align*}
\end{proof}

\begin{notation}
 We write $C(N)$ for the set of permutations $\pi:\{1,\ldots, N\}\to\{1,\ldots,N\}$ that are of the form
\begin{equation*}
 \pi(n)=n+k\mod N
\end{equation*}
for some $k\in\NN$.
\end{notation}

Now, if we are given a modulus $\hat\rho$ satisfying \eqref{eq:rho}, we define $\rho:\NN\times(0,\infty)\to\infty$ by $\rho(d,\varepsilon):=\hat\rho(d,\varepsilon/(2N+1))$. In light of the previous theorem, this new modulus $\rho$ will then satisfy for any $\pi\in C(N)$ the implication
\begin{equation}\label{eq:rhoperm}
  \Vert x-p\Vert\le d\text{ and }\Vert T_{\pi(N)}\cdots T_{\pi(1)}x-x\Vert<\rho(d,\varepsilon)\text{ imply }\Vert T_ix-x\Vert<\varepsilon,\text{ for all }1\le i\le N.
\end{equation}

Observe that if all $T_i$ are also strongly quasi-nonexpansive (SQNE) in the sense of Bruck \cite{Bruck(82)}, then one can transform an SQNE-modulus in the sense of Kohlenbach \cite{Kohlenbach(15)} into a function $\rho$ satisfying \eqref{eq:rho}:

\begin{proposition}[see \cite{Kohlenbach(15)}]\label{prop:Koh}
 Let $(X,d)$ be a metric space and $S\subseteq X$ be a subset. Let $T_1,\ldots,T_N$ be SQNE-mappings with SQNE-moduli $\omega_1,\ldots,\omega_N$, respectively, with respect to some common fixed point $p\in S$ of $T_1,\ldots,T_N$ and let $d\in N$. Assume that $T_1,\ldots,T_N$ are uniformly continuous on $S_d:=\{x\in S:d(x,p)\le d\}$ with modulus of continuity $\alpha:(0,\infty)\to(0,\infty)$, i.e.~for all $\varepsilon>0$ and all $y,y'\in S_d$,
 \begin{equation*}
  d(y,y')<\alpha(\varepsilon)\text{ implies }d(T_iy,T_iy')<\varepsilon\text{ for all }1\le i\le N.
 \end{equation*}
 For $\omega(d,\varepsilon):=\min\limits_{1\le i\le N}\omega_i(d,\varepsilon)$, define
 \begin{equation*}
  \chi_d(0,\varepsilon):=\min\{\alpha(\varepsilon/2),\varepsilon\},\quad\chi_d(n+1,\varepsilon):=\min\left\{\omega\left(d,\frac12\chi_d(n,\varepsilon)\right),\frac12\chi_d(n,\varepsilon)\right\}.
 \end{equation*}
 Then $\rho(d,\varepsilon):=\chi_d(N-1,\varepsilon)$ satisfies for all $x\in C$ and all $\varepsilon>0$
 \begin{equation*}
  d(x,p)\le d\text{ and }d(T_NT_{N-1}\cdots T_1 x,x)<\rho(d,\varepsilon)\text{ imply }d(T_ix,x)<\varepsilon\text{ for all }1\le i\le N.
 \end{equation*}

\end{proposition}
Observe that, if the $T_i$ are SQNE and nonexpansive, then the identity on $(0,\infty)$ is a modulus of continuity $\alpha$ in the sense of the proposition above.
 
 Consider now the following iteration scheme (see e.g.~\cite{Yamada(01)})
 \begin{equation}\label{eq:algoMult}
  u_{n+1}:=T^{(\lambda_{n+1})}_{[n+1]}(u_n):=(1-\lambda_{n+1})T_{[n+1]}(u_n)+\lambda_{n+1}\G T_{[n+1]}(u_n),
 \end{equation}
 where $(\lambda_n)\subset(0,1]$ and $[n]:=n\mod N$.

 \begin{lemma}\label{lem:verylast}
  Suppose the closed convex set $C\subseteq H $ of a Hilbert space $H$ is bounded with $\diam(C)\le d$ for all nonnegative integers $n$ and $\chi:\NN\to\NN$ is a rate of convergence for $(\lambda_n)$ to $0$, i.e.~$\lambda_n\le1/k$ for all nonnegative integers $n\ge\chi(k)$. Then 
  \begin{equation*}
   \Vert u_{n+1}-T_{[n+1]}(u_n)\Vert\le\frac1k,\quad\text{for all nonnegative integers }n\ge\chi\left(d\cdot k\right).
  \end{equation*}
 \end{lemma}
 
 \begin{proof}
  Follows immediately from \eqref{eq:algoMult}.
 \end{proof}

 \begin{theorem}\label{thm:asy}
  Suppose $C$ is bounded with $\diam C\le d$ and $\chi$ is as before. Given moduli $\phi_3:(0,\infty)\times\NN\to\NN$ and $\phi_4:(0,\infty)\to\NN$ such that
  \begin{enumerate}
   \item $\phi_3(\varepsilon,n)\ge n$ for all $\varepsilon>0$ and all $n\in\NN$,
   \item $\prod_{i=n}^{m}(1-\lambda_i(1-\tau))\le\varepsilon\quad$ for all nonnegative integers $n,m$ with $m\ge\phi_3(\varepsilon,n)$, and
   \item $\sum_{i=\phi_4(k)}^\infty\vert\lambda_{i+N}-\lambda_i\vert\le\varepsilon$ for all $\varepsilon>0$.
  \end{enumerate}
  Then, for all $\varepsilon>0$ and all $n\ge\hat\chi(\varepsilon):=\max\{\phi_3(\varepsilon/2d,\phi_4(\varepsilon/4d)),\chi(\lceil Nd/2\varepsilon\rceil)\}$,
  \begin{equation*}
   \Vert u_{n}-T_{[n+N]}\cdots T_{[n+1]}(u_n)\Vert\le\varepsilon.
  \end{equation*}
 \end{theorem}
 
 Using Theorem \ref{thm:perm}, this theorem immediately implies the asymptotic regularity of $(u_n)$ with respect to the mapping $T_{\pi(N)}\cdots T_{\pi(1)}$ for each $\pi\in C(N)$:
 \begin{corollary}\label{thm:asy2}
  In the situation of Theorem \ref{thm:asy}, for all $\pi\in C(n)$, all $\varepsilon>0$ and all $n\ge\hat\chi(\rho(d,\varepsilon/N))$
  \begin{equation*}
   \Vert u_n-T_{\pi(N)}\ldots T_{\pi(1)}u_n\Vert\le\varepsilon.
  \end{equation*}

 \end{corollary}

 \begin{proof}
  Inequality (37) of \cite{Yamada(01)} reads (see also the remarks preceding Lemma \ref{lem:contractive})
  \begin{multline*}
   \Vert u_{n+N}-u_n\Vert\le d\sum_{k=m+1}^n\vert\lambda_{k+N}-\lambda_k\vert+\Vert u_{m+N}-u_m\Vert\prod_{k=m+1}^n(1-\lambda_{k+N}(1-\tau)),\\\text{for all }n>m\ge0.
  \end{multline*}
  Therefore, for $m=\phi_4(\varepsilon/2d)-1$, we get for all $n\ge\phi_4(\varepsilon/2d)-1$
  \begin{align*}
   \Vert u_{n+N}-u_n\Vert&\le d\sum_{k=\phi_4(\varepsilon/2d)}^\infty\vert\lambda_{k+N}-\lambda_k\vert+d\prod_{k=\phi_4(\varepsilon/2d)}^n(1-\lambda_{k+N}(1-\tau))\\
   &\le\frac{\varepsilon}{2}+d\prod_{k=\phi_4(\varepsilon/2d)}^n(1-\lambda_{k+N}(1-\tau)).
  \end{align*}
  Therefore, $\Vert u_{n+N}-u_n\Vert\le\varepsilon$ for all $n\ge\phi_3(\varepsilon/2d,\phi_4(\varepsilon/2d))$.
  
  Now observe that
  \begin{align*}
   u_{n+N}-T_{[n+N]}\cdots T_{[n+1]}(u_n)&=\sum_{k=1}^{N-1}\Big(T_{[n+N]}\cdots T_{[n+N-k+1]}(u_{n+N-k})\\
   &\qquad\qquad\qquad-T_{[n+N]}\cdots T_{[n+N-k]}(u_{n+N-k-1})\Big)\\
   &\quad+u_{n+N}-T_{[n+N]}(u_{n+N-1}).
  \end{align*}
  Therefore, since each $T_i$ is nonexpansive
  \begin{equation*}
   \Vert u_{n+N}-T_{[n+N]}\cdots T_{[n+1]}(u_n)\Vert\le\sum_{k=0}^{N-1}\Vert u_{n+N-k}-T_{[n+N-k]}(u_{n+N-k-1})\Vert.
  \end{equation*}
  Consequently, using Lemma \ref{lem:verylast}, for all $n\ge\max\{\phi_3(\varepsilon/2d,\phi_4(\varepsilon/4d)),\chi(\lceil Nd/2\varepsilon\rceil)\}$,
  \begin{equation*}
   \Vert u_n-T_{[n+N]}\cdots T_{[n+1]}(u_n)\Vert\le\Vert u_n-u_{n+N}\Vert+\Vert u_{n+N}-T_{[n+N]}\cdots T_{[n+1]}(u_n)\Vert\le\varepsilon.
  \end{equation*}

 \end{proof}

We will need the following fact:

\begin{lemma}[see e.g.~Fact 2.13(a) of \cite{Yamada(01)}]\label{lem:last}
 For any real sequence $(\lambda_n)\subset[0,1]$ and nonnegative integers $n$ and $m$ such that $n\ge m$,
 \begin{equation*}
  \sum_{i=m}^n\left(\lambda_i\prod_{j=i+1}^n(1-\lambda_j)\right)\le1.
 \end{equation*}

\end{lemma}

\begin{lemma}\label{lem:core}
 Suppose as before that $C$ is bounded with $\diam C\le d$, where $d\in\NN$. Suppose moreover that $\lambda_n\in[0,1]$, $u^*\in C$, $\phi_3:\NN\to\NN$ and $\chi:\NN\to\NN$ satisfy 
 \begin{enumerate}
  \item $\lambda_n\le1/k$ for all nonnegative integers $n\ge\chi(k)$, and
  \item  $\prod_{i=n}^m(1-\lambda_i(1-\tau))\le\varepsilon$ for all $\varepsilon>0$ and all nonnegative integers $n,m$ with $m\ge\phi_3(\varepsilon,n)$.
 \end{enumerate}
 Then, for all $\varepsilon>0$, $n_0\in\NN$ and all $g:\NN\to\NN$,
 \begin{enumerate}[label=(\roman*)]
  \item\label{item:corefam1} $\displaystyle n_0\ge\chi\left(\left\lceil\frac{12d^2}{\varepsilon^2(1-\tau)}\right\rceil\right)$,
  \item\label{item:corefam2} $\displaystyle\left\langle T_{[n+1]}(u_n)-u^*,\G u^*-u^*\right\rangle\le\frac{\varepsilon^2(1-\tau)}{6}$ for all $n\in[n_0;\tilde g(\phi_3(\varepsilon^2/3d^2),n_0)-1]$, and
  \item\label{item:corefam3} $\displaystyle\Vert T_{[n+1]}(u^*)-u^*\Vert\le\Omega_d(\varepsilon,g,n_0):=\frac{\varepsilon^2}{18d\tilde g(\phi_3(\varepsilon^2/3d,n_0)-n_0)}$ for all\\ $n\in[n_0;\tilde g(\phi_3(\varepsilon^2/3d^2,n_0))-1]$
 \end{enumerate}
 imply $\Vert u_{\tilde g(\phi_3(\varepsilon^2/3d,n_0))}-u^*\Vert\le\varepsilon$, where $\tilde g(n):=\max\{n,g(n)\}$.
\end{lemma}

\begin{proof}
 Observe that 
 \begin{align*}
  \Vert u_{n+1}-u^*\Vert^2&=\left\langle T_{[n+1]}^{(\lambda_{n+1})}(u_n)-u^*,T_{[n+1]}^{(\lambda_{n+1})}(u_n)-u^*\right\rangle\\
			  &=\left\langle T_{[n+1]}^{(\lambda_{n+1})}(u_n)-T_{[n+1]}^{(\lambda_{n+1})}(u^*),T_{[n+1]}^{(\lambda_{n+1})}(u_n)-u^*)\right\rangle\\
			  &\quad+\left\langle T_{[n+1]}^{(\lambda_{n+1})}(u^*)-u^*,T_{[n+1]}^{(\lambda_{n+1})}(u_n)-u^*\right\rangle\\
			  &=\left\Vert T_{[n+1]}^{(\lambda_{n+1})}(u_n)-T_{[n+1]}^{(\lambda_{n+1})}(u^*)\right\Vert^2+\left\langle T_{[n+1]}^{(\lambda_{n+1})}(u^*)-u^*,T_{[n+1]}^{(\lambda_{n+1})}(u_n)-u^*\right\rangle\\
			  &\quad+\left\langle T_{[n+1]}^{(\lambda_{n+1})}(u_n)-T_{[n+1]}^{(\lambda_{n+1})}(u^*),T_{[n+1]}^{(\lambda_{n+1})}(u^*)-u^*\right\rangle\\
			  &=\left\Vert T_{[n+1]}^{(\lambda_{n+1})}(u_n)-T_{[n+1]}^{(\lambda_{n+1})}(u^*)\right\Vert^2\\
			  &\quad+\left\langle 2T_{[n+1]}^{(\lambda_{n+1})}(u_n)-T_{[n+1]}^{(\lambda_{n+1})}(u^*)-u^*,T_{[n+1]}^{(\lambda_{n+1})}(u^*)-u^*\right\rangle\\
			  &=\left\Vert T_{[n+1]}^{(\lambda_{n+1})}(u_n)-T_{[n+1]}^{(\lambda_{n+1})}(u^*)\right\Vert^2\\
			  &\quad+\lambda_{n+1}\left\langle 2\G T_{[n+1]}(u_n)-\G T_{[n+1]}(u^*)-u^*,T_{[n+1]}^{(\lambda_{n+1})}(u^*)-u^*\right\rangle\\
			  &\quad+(1-\lambda_{n+1})\left\langle 2T_{[n+1]}(u_n)-T_{[n+1]}(u^*)-u^*,T_{[n+1]}^{(\lambda_{n+1})}(u^*)-u^*\right\rangle\\
			  &=\left\Vert T_{[n+1]}^{(\lambda_{n+1})}(u_n)-T_{[n+1]}^{(\lambda_{n+1})}(u^*)\right\Vert^2+\underbrace{2\left\langle T_{[n+1]}(u_n)-u^*,T_{[n+1]}^{(\lambda_{n+1})}(u^*)-u^*\right\rangle}_{t_1:=}\\
			  &\quad+\underbrace{\left\langle u^*-T_{[n+1]}(u^*),T_{[n+1]}^{(\lambda_{n+1})}(u^*)-u^*\right\rangle}_{t_2:=}\\
			  &\quad+\lambda_{n+1}\Bigl\langle 2\G T_{[n+1]}(u_n)-2T_{[n+1]}(u_n)+T_{[n+1]}(u^*)-\G T_{[n+1]}(u^*),\\
			  &\quad\phantom{+}\underbrace{\hspace{8.5cm} T_{[n+1]}^{(\lambda_{n+1})}(u^*)-u^*\Bigr\rangle}_{t_3:=}.
 \end{align*}
 Observe that $T_{[n+1]}^{(\lambda_{n+1})}(u^*)-u^*=\lambda_{n+1}(\G T_{[n+1]}(u^*)-u^*)+(1-\lambda_{n+1})(T_{[n+1]}(u^*)-u^*)$. Therefore,
 \begin{align}\label{eq:t_1}
  t_1&=2\left\langle T_{[n+1]}(u_n)-u^*,T_{[n+1]}^{(\lambda_{n+1})}(u^*)-u^*\right\rangle\notag\\
  &=2\lambda_{n+1}\left\langle T_{[n+1]}(u_n)-u^*,\G T_{[n+1]}(u^*)-u^*\right\rangle\notag\\
  &\quad+2(1-\lambda_{n+1})\langle T_{[n+1]}(u_n)-u^*,T_{[n+1]}(u^*)-u^*\rangle\notag\\
  &\le2\lambda_{n+1}\langle  T_{[n+1]}(u_n)-u^*,\G T_{[n+1]}(u^*)-\G u^*\rangle+2\lambda_{n+1}\langle  T_{[n+1]}(u_n)-u^*,\G u^*-u^*\rangle\notag\\
  &\quad+2d(1-\lambda_{n+1})\cdot\Vert T_{[n+1]}(u^*)-u^*\Vert\notag\\
  &\le 2d\lambda_{n+1}\tau\Vert T_{[n+1]}(u^*)-u^*\Vert+2d(1-\lambda_{n+1})\Vert T_{[n+1]}(u^*)-u^*\Vert\notag\\
  &\quad+2\lambda_{n+1}\langle T_{[n+1]}(u_n)-u^*,\G u^*-u^*\rangle\notag\\
  &\le 2d\cdot\Vert T_{[n+1]}(u^*)-u^*\Vert+2\lambda_{n+1}\langle T_{[n+1]}(u_n)-u^*,\G u^*-u^*\rangle.
 \end{align}
 Moreover,
  \begin{align}\label{eq:t_2}
  t_2&=\left\langle u^*-T_{[n+1]}(u^*),T_{[n+1]}^{(\lambda_{n+1})}(u^*)-u^*\right\rangle\notag\\
     &\le d\cdot\Vert u^*-T_{[n+1]}(u^*)\Vert.
 \end{align}
 For the term $t_3$, we get the estimate
 \begin{align}\label{eq:t_3}
  t_3&=\lambda_{n+1}\left\langle 2\G T_{[n+1]}(u_n)-2T_{[n+1]}(u_n)+T_{[n+1]}(u^*)-\G T_{[n+1]}(u^*),T_{[n+1]}^{(\lambda_{n+1})}(u^*)-u^*\right\rangle\notag\\
  &\le2\lambda_{n+1}^2\left\langle \G T_{[n+1]}(u_n)-T_{[n+1]}(u_n),\G T_{[n+1]}(u^*)-u^*\right\rangle\notag\\
  &\quad+2d\lambda_{n+1}(1-\lambda_{n+1})\Vert T_{[n+1]}(u^*)-u^*\Vert\notag\\
  &\quad+\lambda_{n+1}^2\langle T_{[n+1]}(u^*)-\G T_{[n+1]}(u^*),\G T_{[n+1]}(u^*)-u^*\rangle\notag\\
  &\quad+d\lambda_{n+1}(1-\lambda_{n+1})\Vert T_{[n+1]}(u^*)-u^*\Vert\notag\\
  &\le3d^2\lambda_{n+1}^2+3d\lambda_{n+1}(1-\lambda_{n+1})\Vert T_{[n+1]}(u^*)-u^*\Vert\notag\\
  &\le3d^2\lambda_{n+1}^2+3d\lambda_{n+1}\Vert T_{[n+1]}(u^*)-u^*\Vert.
 \end{align}
 Combining the estimates for the terms $t_1$, $t_2$ and $t_3$, we obtain
 \begin{align*}
    \Vert u_{n+1}-u^*\Vert^2&\le\left\Vert T_{[n+1]}^{(\lambda_{n+1})}(u_n)-T_{[n+1]}^{(\lambda_{n+1})}(u^*)\right\Vert^2+(3d\lambda_{n+1}+3d)\Vert T_{[n+1]}(u^*)-u^*\Vert\\
    &\quad+3d^2\lambda_{n+1}^2+2\lambda_{n+1}\langle T_{[n+1]}(u_n)-u^*,\G u^*-u^*\rangle.
 \end{align*}
 Therefore, by Lemma \ref{lem:contractive}
 \begin{align*}
  \Vert u_{n+1}-u^*\Vert^2&\le(1-\lambda_{n+1}(1-\tau))^2\Vert u_n-u^*\Vert^2+(3d\lambda_{n+1}+3d)\Vert T_{[n+1]}(u^*)-u^*\Vert\\
    &\quad+3d^2\lambda_{n+1}^2+2\lambda_{n+1}\langle T_{[n+1]}(u_n)-u^*,\G u^*-u^*\rangle.\\
    &\le(1-\lambda_{n+1}(1-\tau))^2\Vert u_n-u^*\Vert^2+6d\cdot\Vert T_{[n+1]}(u^*)-u^*\Vert\\
    &\quad+3d^2\lambda_{n+1}^2+2\lambda_{n+1}\langle T_{[n+1]}(u_n)-u^*,\G u^*-u^*\rangle.
 \end{align*}
 Using now the hypotheses, we see that, for all $n\in[n_0;\tilde g(\phi_3(\varepsilon^2/3d^2),n_0)-1]$,
 \begin{align*}
  \Vert u_{n+1}-u^*\Vert^2\le(1-\lambda_{n+1}(1-\tau))\Vert u_n-u^*\Vert^2+6d\cdot\Omega_d(\varepsilon,g,n_0)+\lambda_{n+1}(1-\tau)\cdot\frac{\varepsilon^2}{3}.
 \end{align*}
Using induction, we then get (using the abbreviation $\tilde n_0=\tilde g(\phi_3(\varepsilon^2/3d^2),n_0)$) by Lemma \ref{lem:last}
\begin{align*}
 \Vert u_{\tilde n_0}-u^*\Vert^2&\le\Vert u_{{n_0}}-u^*\Vert^2\cdot\prod_{i={n_0}+1}^{\tilde n_0}(1-\tau(1-\lambda_i))+6d(\tilde n_0-n_0)\Omega_d(\varepsilon,g,{n_0})\\
 &\quad+\frac{\varepsilon^2}{3}\sum_{i={n_0}+1}^{\tilde n_0}\left(\lambda_i(1-\tau)\prod_{j=i+1}^{\tilde n_0}(1-\lambda_j(1-\tau))\right)\\
 &\le\Vert u_{n_0}-u^*\Vert^2\cdot\prod_{i={n_0}+1}^{\tilde n_0}(1-\tau(1-\lambda_i))+\frac{\varepsilon^2}{3}+6d(\tilde n_0-n_0)\Omega_d(\varepsilon,g,{n_0})\\
 &\le\frac{\varepsilon^2}3+\frac{\varepsilon^2}{3}+\frac{\varepsilon^2}{3}=\varepsilon^2.
\end{align*}
\end{proof}

\begin{theorem}\label{thm:mainquant2}
 Suppose $C$ is bounded with $\diam(C)\le d$, and suppose that $T_1,\ldots,T_N:C\to C$ are nonexpansive mappings with a common fixed point $p\in C$ that satisfy $\bigcap_{i=1}^N\fix(T_i)=\fix(T_N\cdots T_1)$. Suppose $\rho:\NN\times(0,\infty)\to(0,\infty)$ satisfies \eqref{eq:rhoperm} and let the moduli $\chi$, $\hat\chi$, $\phi_3$ and $\phi_4$ be as before.
 Then, for any $\tau$-contraction $\G:C\to C$, the iteration given by \eqref{eq:algoMult} is metastable with rate $\Xi(\varepsilon,g;\chi,\phi_3,\phi_4,d,\tau)$, i.e.
 \begin{equation*}
  \forall\varepsilon>0\,\forall g:\NN\to\NN\,\exists n\le\Xi(\varepsilon,g;\chi,\phi_3,\phi_4,d,\tau)\left(\Vert u_n-u_{\tilde g(n)}\Vert\le\varepsilon\right)
 \end{equation*}
where $\tilde g(n):=\max\{n,g(n)\}$ and $\Xi$ is defined by
 \begin{equation*}
  \Xi(\varepsilon,g;\chi,\phi_3,\phi_4,d,\tau):=\phi_3'\left(\varepsilon^2/12d^2,\max\left\{n_0,\hat\chi\left(\rho\left(d,\frac{1}{k_{i_0}(\tilde f)N}\right)\right)\right\}\right),
 \end{equation*}
 where $\phi_3'(\varepsilon,i):=\max\{n,\max\{\phi_3(\varepsilon,i):i\le n\}\}$ and $k_{i_0}(\tilde f)$ is as in Theorem \ref{thm:main} except for $f$ and $\varepsilon_d$, which are now defined as
 \begin{align*}
     f(k)&:=\rho\left(d,\Omega_d^M(\varepsilon/2,\tilde g^M,\max\{n_0,\hat\chi(\rho(d,1/Nk))\}\right)\\
     n_0&:=\max\{\chi(\lceil96d/(1-\tau)\varepsilon^2\rceil),\chi(\lceil48d^2/(1-\tau)\varepsilon^2\rceil)\}\\
   \varepsilon_d&:=\frac{((1-\tau)\varepsilon^2/96)^2}{2d^2}\\
   \Omega_d^M(\varepsilon,g,n)&:=\max\left\{\Omega_d(\varepsilon,g,i):i\le n\right\}.
  \end{align*}
\end{theorem}

\begin{proof}
 Define $T:=T_N\cdots T_1$ and 
 \begin{equation*}
  J_{\varepsilon}(\varphi):=\min\left\{l\ge n_0:\Vert Tu_k-u_k\Vert\le\varphi(v)\text{ for all }v\in C\text{ and }k\ge l\right\}.
 \end{equation*}
 For majorizable $\varphi:C\to(0,1]$, this is well-defined, and by Corollary \ref{thm:asy2}, 
 \begin{equation}\label{eq:jj}
  k\maj\varphi\ra J_{\varepsilon}(\varphi)\le\max\left\{n_0,\hat\chi(\rho(d,1/Nk)\right\}.
 \end{equation}
 Now define the counterfunction $V_{\varepsilon,g}(u,\varphi)=u_{i(u)+1}$, where $i(u)$ is defined as the least index $i\in[J_{\varepsilon}(\varphi)-1,\tilde g_{u,\varepsilon}(\phi_3(\varepsilon^2/12d^2,J_{\varepsilon}(\varphi)))-2]$ such that for all integers $k\in[J_{\varepsilon(\varphi)}-1,\tilde g_{u,\varepsilon}(\phi_3(\varepsilon^2/12d^2,J_{\varepsilon}(\varphi))-2]$
 \begin{equation*}
  \Vert\G u-(1-t)u-tu_i\Vert\le\Vert\G u-(1-t)u-tu_k\Vert.
 \end{equation*}
 Moreover, if we define $\Delta_{\varepsilon,g}(u,\varphi):=\rho\left(d,\Omega_d(\varepsilon/2,\tilde g_{u,\varepsilon},J_{\varepsilon}(\varphi))\right)$, then, given a majorant $k\maj\varphi$, \eqref{eq:jj} implies $\rho\left(d,\Omega_d^M\left(\varepsilon/2,\tilde g^M,\max\{n_0,\hat\chi\left(\frac{1}{k}\right)\}\right)\right)\ge\Delta_{\varepsilon,g}(u,\varphi)$ for all $u\in C$. Therefore, $f\maj\Delta_{\varepsilon,g}$. We again write $\tilde u:=U(\varepsilon_d,t,\Delta_{\varepsilon,g},V_{\varepsilon,g})$ and $\tilde\varphi:=\Phi(\varepsilon_d,t,\Delta_{\varepsilon,g},V_{\varepsilon,g})$, where $t:=\frac{((1-\tau)\varepsilon^2/96)^2}{3d^2}$.
 
 Then,
 \begin{equation}\label{eq:fix3}
  \Vert T\tilde u-\tilde u\Vert<\rho\left(d,\Omega_d(\varepsilon/2,\tilde g_{\tilde u,\varepsilon},J_{\varepsilon}(\tilde\varphi))\right)
 \end{equation}
 and
 \begin{equation*}
  \Vert Tu_{i(\tilde u)+1}-u_{i(\tilde u)+1}\Vert<\tilde\varphi(u_{i(\tilde u)+1})\ra\Vert\G\tilde u-\tilde u\Vert^2<\Vert\G\tilde u-(1-t)\tilde u-tu_{i(\tilde u)+1}\Vert^2+\varepsilon_d
 \end{equation*}
 By construction, $\Vert Tu_{i(\tilde u)+1}-u_{i(\tilde u)+1}\Vert<\tilde\varphi(u_{i(\tilde u)+1})$, so we conclude $\Vert\G\tilde u-\tilde u\Vert^2<\Vert\G\tilde u-(1-t)\tilde u-tu_{i(\tilde u)+1}\Vert^2+\varepsilon_d$. By construction, we also have $\Vert\G\tilde u-(1-t)\tilde u-tu_{i(\tilde u)+1}\Vert\le\Vert\G\tilde u-(1-t)\tilde u-tu_{k+1}\Vert$ for all $k\in[J_{\varepsilon}(\tilde\varphi),\tilde g_{\tilde u,\varepsilon}(\phi_3(\varepsilon^2/12d^2,J_{\varepsilon}(\tilde\varphi))-1]$. Therefore,
 \begin{equation*}
  \Vert\G\tilde u-\tilde u\Vert^2<\Vert\G\tilde u-(1-t)\tilde u-tu_{k+1}\Vert^2+\varepsilon_d,\quad\text{for all }k\in[J_{\varepsilon}(\tilde\varphi),\tilde g_{\tilde u,\varepsilon}(\phi_3(\varepsilon^2/12d^2,J_{\varepsilon}(\tilde\varphi))-1].
 \end{equation*}
 Therefore, Lemma \ref{lem:switch} implies $\left\langle\G\tilde u-\tilde u,u_{k+1}-\tilde u\right\rangle\le(1-\tau)\varepsilon^2/96$ for all nonnegative integers $k\in[J_{\varepsilon}(\tilde\varphi),\tilde g_{\tilde u,\varepsilon}(\phi_3(\varepsilon^2/12d^2,J_{\varepsilon}(\tilde\varphi))-1]$. Consequently, since by construction\\ $J_{\varepsilon}(\tilde\varphi)\ge\max\left\{\chi(\lceil96d/(1-\tau)\varepsilon^2\rceil),\chi(\lceil48d^2/(1-\tau)\varepsilon^2\rceil)\right\}$,
 \begin{equation*}
  \langle\G\tilde u-\tilde u,T_{[k+1]}(u_k)-\tilde u\rangle\le\langle\G\tilde u-\tilde u,u_{k+1}-\tilde u\rangle+d\cdot\Vert T_{[k+1]}(u_k)-u_{k+1}\Vert\le\frac{(\varepsilon/2)^2(1-\tau)}{12},
 \end{equation*}
 and
 \begin{equation*}
  J_{\varepsilon}(\tilde\varphi)\ge\chi\left(\left\lceil\frac{12}{(\varepsilon/2)^2(1-\tau)}\right\rceil\right).
 \end{equation*}
 Moreover, \eqref{eq:rhoperm} and \eqref{eq:fix3} imply $\Vert T_{[k+1]}(\tilde u)-\tilde u\Vert<\Omega_d^M(\varepsilon/2,\tilde g_{\tilde u,\varepsilon},J_{\varepsilon}(\tilde\varphi))$ for all nonnegative integers $k$. Therefore, Lemma \ref{lem:core} implies
 \begin{equation*}
  \Vert u_{\tilde g_{\tilde u,\varepsilon}(\phi_3(\varepsilon^2/12d^2,J_{\varepsilon}(\tilde\varphi))}-\tilde u\Vert\le\frac{\varepsilon}{2}.
 \end{equation*}
 Therefore, for $k:=\phi_3(\varepsilon^2/12d^2,J_{\varepsilon}(\tilde\varphi))$, Lemma \ref{lem:subseq} yields
 \begin{equation*}
  \Vert u_k-u_{\tilde g(k)}\Vert\le\varepsilon.
 \end{equation*}
\end{proof}

As before, one can weaken the assumption that $C$ is bounded as follows:
\begin{corollary}\label{cor:analogo}
 For all of the results in this section, one can drop the condition of $C$ being bounded with $\diam(C)\le d$ in favor of  $\Vert u_0-v\Vert\le d/4$, $\Vert v-\G v\Vert\le\frac{d(1-\tau)}{4}$ and $\Vert v-w\Vert\le\frac{d}{4(1+\tau)}$, where $v$ is a common fixed point of the $T_i$ and $w$ is the unique fixed point of $\G$.
\end{corollary}

\begin{proof}
 Similarly to the situation before, Lemma \ref{lem:contractive} implies for all nonnegative integers $n$
 \begin{align*}
  \Vert u_{n+1}-v\Vert&\le\Vert T_{[n+1]}^{(\lambda_{n+1})}(u_n)-T_{[n+1]}^{(\lambda_{n+1})}(v)\Vert+\Vert T_{[n+1]}^{(\lambda_{n+1})}(v)-v\Vert\\
		      &\le(1-\lambda_{n+1}(1-\tau))\Vert u_n-v\Vert+\lambda_{n+1}(1-\tau)\cdot\frac{\Vert \G v-v\Vert}{1-\tau}.
 \end{align*}
 Since $\Vert u_0-v\Vert\le d/4$, we conclude by induction that $\Vert u_n-v\Vert\le d/4$. Moreover,
 \begin{equation*}
  \Vert\G u_n-v\Vert\le\Vert \G u_n-\G w\Vert+\Vert w-v\Vert\le\tau\Vert u_n-w\Vert+\Vert v-w\Vert\le\tau\Vert u_n-v\Vert+(1+\tau)\Vert v-w\Vert.
 \end{equation*}
 Consequently, both $(u_n)$ and $(\G u_n)$ remain in the closed ball of radius $d/2$, and hence of diameter $d$, centered at $v$. Since all points for which the condition $\diam(C)$ were either elements of the sequences $(u_n)$ and $(\G u_n)$, or convex combinations thereof, the claim follows.
\end{proof}

\noindent{\textbf{Acknowledgement:} I wish to thank Professor Genaro López Acedo for his invitation to the ``Instituto Universitario de Investigación de Mathemáticas'' at the University of Seville and the numerous insightful discussions.

Moreover, I wish to thank Professor Ulrich Kohlenbach for suggesting this paper as the subject of a proof-theoretic analysis.

\end{document}